\UseRawInputEncoding
\documentclass[11pt]{article}
\usepackage{amssymb,amsmath}
\usepackage{graphics}
\usepackage{float}
\usepackage{booktabs}
\usepackage{subfigure}
\usepackage{color,fullpage}

\usepackage{cite}
\usepackage{bm}
\usepackage[title,titletoc]{appendix}
\usepackage{algorithm}%
\usepackage{algorithmicx}%
\usepackage{multirow}
\usepackage{pgfplots}
\usepackage{amsthm}
\usepackage{enumerate}

\newtheorem{assumption}{Assumption}[section]

\newtheorem{theorem}{Theorem}[section]
\newtheorem{lemma}{Lemma}[section]

\newtheorem{remark}{Remark}[section]
\newtheorem{proposition}{Proposition}[section]

\makeatletter

\@addtoreset{equation}{section} \makeatother

\begin{document}

\title{Relaxed Greedy Randomized Kaczmarz with Signal Averaging for Solving Doubly-Noisy Linear Systems}

\author{
Lu Zhang\thanks{College of Science, National University of Defense Technology,
	Changsha, Hunan 410073, China.  Email: \texttt{zhanglu21@nudt.edu.cn}}
\and Jinchuan Zeng\thanks{College of Science, National University of Defense Technology,
	Changsha, Hunan 410073, People's Republic of China. Email: \texttt{zengjinchuan23@nudt.edu.cn}}
\and Hui Zhang\thanks{Corresponding author.
College of Science, National University of Defense Technology,
Changsha, Hunan 410073, China.  Email: \texttt{h.zhang1984@163.com}
}
}


\date{\today}

\maketitle

\begin{abstract}	
Large-scale linear systems of the form $Ax=b$ are often doubly-noisy, in the sense that both its measurement matrix $A$ and measurement vector $b$ are noisy.
In this paper, we extend the relaxed greedy randomized Kaczmarz (RGRK) method to the doubly-noisy systems to accelerate convergence. However, RGRK fails to converge
to the least-squares solution for doubly-noisy systems. 
To address this limitation, we propose a simple modification: averaging multiple measurements instead of using a single measurement. The proposed RGRK with signal averaging (RGRK-SA) converges to the solution of doubly-noisy systems at a polynomial rate. 
Numerical experiments demonstrate that both RGRK and RGRK-SA outperform the classical randomized Kaczmarz method, and RGRK-SA has a higher accuracy.
\end{abstract}


\section{Introduction}
Finding solutions to large-scale linear systems is one of the most fundamental problems in linear algebra and diverse engineering disciplines. 
Many practical problems can be cast in the form
\begin{equation}
\label{eq1.1}
Ax=b,
\end{equation}
where $A\in\mathbb{R}^{m\times n}$ and $b\in\mathbb{R}^{m}$ are given and $x\in\mathbb{R}^n$ is unknown. 
Throughout the paper, we assume that the system (\ref{eq1.1}) is consistent with the least-squares solution $\hat{x}=A^{\dagger}b$, which guarantees (\ref{eq1.1}) has at least one solution.
The Kaczmarz method \cite{karczmarz1937angenaherte} is a well-known iterative projection algorithm for solving the linear system (\ref{eq1.1}), owing to its low computational cost and geometric interpretation. To improve the convergence rate, \cite{strohmer2009randomized} proposed the randomized Kaczmarz (RK) method, which attains linear convergence in expectation by randomly selecting the rows of $A$.
The elegant convergence theory of RK has sparked extensive research interests, particularly in developing sampling techniques
\cite{zhang2019new,bai2018greedy,bai2018relaxed,zhang2022greedy,niu2020greedy,zhang2022weighted,tondji2021linear,yuan2022sparse,chen2023greedy,zhang2025quantile,yuan2022adaptively}.

However, the exact data $A$ and $b$ are generally unknown in practical applications. It is common that we are only given noisy measurements $\tilde{A}\in\mathbb{R}^{m\times n}$ and $\tilde{b}\in\mathbb{R}^{m}$. Thus, we turn our attention to consider the 
doubly-noisy linear system
\begin{equation}
\label{eq1.2}
\tilde{A}x\approx \tilde{b}.
\end{equation} 
Here we consider additive noise in the sense that
$\tilde{A}=A+E,~\tilde{b}=b+\varepsilon$, where $E\in\mathbb{R}^{m\times n}$ and $\varepsilon\in\mathbb{R}^{m}$ are the noise of $A$ and $b$, respectively. Note that the doubly-noisy system (\ref{eq1.2}) is not necessarily consistent.
How to reconstruct the solution to the doubly-noisy systems turns into a practical challenge.

Very recently, the authors of \cite{bergou2024note} extended RK to solve the doubly-noisy system (\ref{eq1.2}) with iterative scheme
\begin{equation}\label{eq1.3}
x_{k+1}=x_k-\frac{	\tilde{a}_{i_k}^Tx_k-\tilde{b}_{i_k}}{\|\tilde{a}_{i_k}\|^2_2}\tilde{a}_{i_k},\end{equation}
where $\tilde{a}_{i_k}$ represents the $i_k$-th row of the matrix $\tilde{A}\in\mathbb{R}^{m\times n}$ and $\tilde{b}_{i_k}$ is the $i_k$-th element of the vector $b\in\mathbb{R}^m$.
In the $k$-th iterate of RK, the current iterate $x_k$ is orthogonally projected onto the hyperplane $H_k=\{x\in\mathbb{R}^n\mid \langle \tilde{a}_{i_k},x\rangle=\tilde{b}_{i_k}\}$ sampled with probability 
\begin{equation}
\label{eq1.5}
P(i=i_k)=\frac{\|\tilde{a}_{i_k}\|^2}{\|\tilde{A}\|_F^2}.
\end{equation}
The convergence result of the scheme (\ref{eq1.3}) in \cite{bergou2024note} is given by
\begin{equation}\label{eq1.4}
\mathbb{E}\|x_{k+1}-\hat{x}\|^2
\leq 
\left(1-\frac{\sigma_{\min}^2(\tilde{A})}{\|\tilde{A}\|_F^2}
\right)
\mathbb{E}\|x_k-\hat{x}\|_2^2
+
\frac{\|E\hat{x}-\varepsilon\|_2^2}{\|\tilde{A}\|_F^2}.
\end{equation}
When extended to doubly-noisy systems, the RK method suffers from two limitations. First, its sampling method, as in \cite{strohmer2009randomized}, lacks geometric interpretation and does not utilize iterative information. Second, as presented in (\ref{eq1.4}), the convergence horizon $\frac{\|E\hat{x}-\varepsilon\|_2^2}{\|\tilde{A}\|_F^2}$ caused by additive noise $E$ and $\varepsilon$ prevents the iterates from converging to the exact solution $\hat{x}$.
These challenges motivate us to modify the RK method by incorporating a more effective sampling method and seeking to eliminate the convergence horizon in (\ref{eq1.4}).

To this end, we first extend the relaxed greedy randomized Kaczmarz (RGRK) method in \cite{bai2018relaxed} to doubly-noisy systems. The relaxed greedy sampling strategy combines the advantages of the greedy rule and randomized sampling method, aiming to achieve significant progress in each iteration. 
We show that RGRK is robust to doubly-noisy systems, and its linear convergence up to an error horizon is established by
\begin{equation}
\label{eq1.6}
\mathbb{E}\|	x_{k+1}-\hat{x}\|^2
\leq 
\left[1-\left(\frac{\theta}{\gamma}\|\tilde{A}\|_F^2+(1-\theta)\right)\cdot\frac{\tilde{\sigma}_{\min}^2(\tilde{A})}{\|\tilde{A}\|_F^2}\right]
\mathbb{E}\|x_k-\hat{x}\|_2^2
+
\frac{\|E\hat{x}-\varepsilon\|_2^2}{ (\|\tilde{A}\|_F-\frac{\gamma}{\|\tilde{A}\|_F})^2},
~k\in\mathbb{N}^{+}.
\end{equation}
Numerical experiments demonstrate that the relaxed greedy sampling method significantly outperforms the scheme in
(\ref{eq1.5}) in terms of both iteration numbers and computing time. However, the finite convergence horizon in (\ref{eq1.6}) adversely affects the numerical performance of RGRK.

To reduce the convergence horizon in (\ref{eq1.6}), we consider the case that the linear system can be measured repeatedly and the true solution remains unchanged during measurements. Following \cite{lyons1997understanding,Dual2023,hassan2010reducing,tondji2024adaptive}, we integrate the signal averaging technique into RGRK to decrease the data error. Specifically, instead of the exact data $A$ and $b$, we are only given multiple inexact measurements $A^1,\ldots,A^N$ and $b^1,\ldots,b^N$, which serve as unbiased estimators of $A$ and $b$, respectively. At $k$-th iterate, the single observation in (\ref{eq1.3}) is replaced by the averaged estimators  
$$
\bar{a}_{i_k}=\frac{1}{N}\sum_{j=1}^{N} a_{i_k}^j,~\bar{b}_{i_k}=\frac{1}{N}\sum_{j=1}^{N} b_{i_k}^j,
$$
which leads to the iterative scheme
\begin{equation}
\label{eq1.7}
x_{k+1}=x_k-\frac{	\bar{a}_{i_k}^Tx_k-\bar{b}_{i_k}}{\|\bar{a}_{i_k}\|^2_2}\bar{a}_{i_k}.
\end{equation}
Combining with the relaxed greedy sampling method, we propose the RGRK with signal averaging (RGRK-SA) method, for which we establish the following convergence result 
\begin{equation}
\label{eq1.8}
\begin{aligned}
&~~~~\mathbb{E}\|	x_{k+1}-\hat{x}\|^2\\
&\leq 
\left[1-\left(\frac{\theta}{\gamma}\|\tilde{A}\|_F^2+(1-\theta)\right)\cdot\frac{\tilde{\sigma}_{\min}^2(\tilde{A})}{\|\tilde{A}\|_F^2}\right]
\mathbb{E}\|x_k-\hat{x}\|_2^2
+
\frac{1}{N}\cdot\frac{\|E\hat{x}-\varepsilon\|_2^2}{ (\|\tilde{A}\|_F-\frac{\gamma}{\|\tilde{A}\|_F})^2},
~k\in\mathbb{N}^{+}.
\end{aligned}
\end{equation}
It follows from (\ref{eq1.8}) that the signal averaging technique reduces the convergence horizon in (\ref{eq1.6}) by a factor of $\frac{1}{N}$. Hence, as the number of measurements $N$ increases, RGRK-SA converges to the least-squares solution of the underlying system (\ref{eq1.1}). 
Numerical experiments on both simulated and real-world data demonstrate that RGRK-SA achieves higher accuracy than RK in \cite{bergou2024note}, which validates our theoretical findings.

\subsection{Contributions}
Our main contributions can be summarized as follows.
\begin{enumerate}[(i)]
\item 
We extend the relaxed greedy randomized Kaczmarz (RGRK) method to doubly-noisy linear systems and establish linear convergence up to a finite error horizon (see Theorem \ref{th1}).
\item
We propose RGRK with signal averaging (RGRK-SA), which incorporates repeated noisy measurements to reduce the convergence horizon of RGRK by a factor of $1/N$ (see Remark \ref{remark:3} (ii)). We further prove that RGRK-SA converges to the least-squares solution of the underlying system at a polynomial rate about the number of measurements (see Theorem \ref{th2}). 
\item 
Numerical experiments demonstrate that RGRK and RGRK-SA exhibit similar convergence speeds, both faster than the standard RK. While RK and RGRK achieve comparable recovery accuracy, RGRK-SA attains the highest solution accuracy, which improves as the number of measurements increases.
\end{enumerate}

\subsection{Organization}
The paper is organized as follows. Section \ref{sec6} presents related works about sampling methods and denoising methods. In section \ref{sec2}, we describe the RGRK method for doubly-noisy systems and then establish its convergence.
In Section \ref{sec3}, we propose the RGRK with signal averaging method to eliminate the convergence horizon caused by double noise.
In Section \ref{sec4}, we carry out some numerical experiments to verify the effectiveness of RGRK and RGRK-SA compared with RK. Finally, Section \ref{sec5} is our conclusion.

\section{The related work}
\label{sec6}
\subsection{The sampling methods}

The sampling method plays a crucial role in studying the convergence rate of RK.
Since the first randomized sampling method was introduced by \cite{strohmer2009randomized}, numerous randomized sampling strategies have been proposed to improve the convergence rate of RK, such as \cite{zhang2022weighted,Steinerberger2021,tondji2021linear,yuan2022adaptively}.
Greedy sampling methods, including the maximum residual \cite{Griebel2012} and maximum distance \cite{Du2019} are widely used. 
Combining the advantages of randomized rules and greedy rules, the sampling Kaczmarz-Motzkin (SKM) method \cite{yuan2022sparse,Zhang2023} and other variants have been developed \cite{miao2022greedy,niu2020greedy,bai2018greedy,bai2018relaxed,zhang2022weighted,Steinerberger2021,tondji2021linear,yuan2022adaptively}.
Another extension of RK is the block Kaczmarz algorithm, which utilizes minibatches to enable parallel computation at each iteration.
The projective block \cite{liu2021greedy,miao2022greedy,niu2020greedy,Needell2012,lorenz2025minimal} make projection onto a subsystem of (\ref{eq1.1}) instead of a single hyperplane, while the averaged block variants \cite{necoara2019faster,moorman2021randomized,zhang2025quantile} make use of distributed computing.

\subsection{The denoising techniques}
\label{sub2.2}
When only noisy measurements are available, i.e., $\tilde{b}=b+\varepsilon$ with noise $\varepsilon$, the RK method no longer converges to the solution but only to a neighborhood around the solution. This limitation has motivated extensive research to reduce the effect of noise on the convergence accuracy. 

The randomized extended Kaczmarz (REK) method was proposed in \cite{zouzias2013randomized}; it alternates projections onto row and column spaces of the measurement matrix and converges to the minimum $l_2$-norm least squares solution of an inconsistent linear system. 
Later, \cite{schopfer2022extended} further generalized REK by incorporating Bregman projections, thereby enabling the recovery of sparse least squares solutions to inconsistent systems. 
Although REK achieves linear convergence, it requires one additional computation and the storage of the entire system due to column-based projection steps, which is expensive for large-scale linear systems.
The randomized Kaczmarz with averaging (RKA) in \cite{moorman2021randomized} performs an average over multiple independent updates at each iteration, which can be seen as minibatch SGD. As the number of threads increases, the convergence horizon of RKA decreases for noisy systems.  
Recently, the tail-averaged randomized Kaczmarz
(TARK) method was proposed in \cite{epperly2024randomized}, which averages the RK interations in the tail part and achieves polynomial convergence for inconsistent systems. 
Very recently, \cite{marshall2023optimal} considered linear systems with independent mean zero noise $\varepsilon$ and showed that using an optimized relaxed stepszie can break through the convergence horizon induced by noise. 
These studies have concentrated on eliminating the convergence horizon caused by noise in the measurement vector $b$, while the convergence horizon of doubly-noisy systems remains to be addressed.

This paper focuses on eliminating the convergence horizon in doubly-noisy linear systems, in which only the noisy measurements $\tilde{A}$ and $\tilde{b}$ are available. Since neither REK nor the optimal relaxed stepsize extends naturally to doubly-noisy systems, we adopt the signal averaging technique, a standard noise reduction technique widely used in medical imaging \cite{baumann2019signal,eichner2015real,cochrane2012performance}, electrocardiography (EEG) \cite{jarrett1991signal}, and electroencephalography (ECG) \cite{jesus1988high}. Notably, the signal averaging technique has also been employed in the numerical experiments of \cite{tondji2024adaptive} to reduce the standard deviation of the noisy measurement vectors $\tilde{b}$. Therefore, we attempt to extend the signal averaging technique to solve doubly-noisy systems.

\subsection{Notation}

For the matrix $A=(a_{ij})\in\mathbb{R}^{m\times n}$, its Frobenius norm is denoted by $\|A\|_F=\sqrt{\sum_{i=1}^{m}\sum_{j=1}^{n} a_{ij}^2}$. Let $A_I$ be the submatrix consisting of the rows of $A$ indexed by the set $I$.
The smallest singular value of $A$ is denoted by $\sigma_{\min}(A)$.
Let $\tilde{\sigma}_{\min}(A)$ denote the smallest singular value among all row submatrices $A_I,~I\subset\{1,\ldots,m\}$ of $A$, that is,
$$\tilde{\sigma}_{\min}(A)=\{ \sigma_{\min}(A_I)\mid I\subset\{1,\ldots,m\},~A_I\neq 0\}.$$
We denote the rows of the noisy matrix $E\in\mathbb{R}^{m\times n}$ by $E_1,\ldots,E_n$ and the elements of the noisy vector $\varepsilon\in\mathbb{R}^m$ by $\varepsilon_1,\ldots,\varepsilon_m$.

For any vectors $x,y\in\mathbb{R}^{n}$, we use $\langle x,y\rangle$ to represent its inner product.
We indicate by $\mathbb{E}_k$ the conditional expectation conditioned on the first $k$ iterations in the sense that
$$\mathbb{E}_k[\cdot]=\mathbb{E}[\cdot\mid i_0,i_1,\ldots,i_{k-1}],$$
where $i_l~(l=0,1,\ldots,k-1)$ is the index of the sampled hyperplane at the $l$-th iterate. Unless otherwise stated, $\|\cdot\|$ denotes the Euclidean norm (2-norm).

\section{The RGRK method for doubly-noisy linear systems}
\label{sec2}

In this paper, we consider the doubly-nosiy systems
\begin{equation}
\label{eq3:2}
\tilde{A}x\approx \tilde{b}, ~\tilde{A}=A+E,~\tilde{b}=b+\varepsilon
\end{equation}
where $A\in\mathbb{R}^{m\times n}$ and $b\in\mathbb{R}^{m}$ represent the exact data.  
The relaxed greedy randomized Kaczmarz (RGRK) method was originally proposed in \cite{bai2018relaxed} for exact linear systems, where it recovers the solution  more efficiently than the standard randomized Kaczmarz by introducing an effective relaxed greedy probability criterion.
The relaxed greedy sampling method exploits the observation that larger residuals often indicate greater potential for progress, and it has been shown that RGRK converges significantly faster than RK when applied to exact system. Nevertheless, there is no research for the more practical settings where both the measurement matrix and the measurement vector are corrupted by noise.

Motivated by its performance on exact systems, we investigate how RGRK can be extended to handle the more challenging doubly-noisy systems.
To this end, the RGBK method applied to the system (\ref{eq1.2}) can be algorithmically described in Algorithm \ref{al1}, where $\theta\in [0,1]$ is a positive parameter and $M$ is the number of iterations.
It follows from
$$\max_{1\leq i\leq m}\left\{
\frac{\lvert \tilde{a}_{i}^Tx_k-\tilde{b}_{i}\rvert^2}{\|\tilde{a}_i\|_2^2}
\right\}
\geq
\sum_{i=1}^{m}\frac{\|\tilde{a}_{i_k}\|^2}{\|\tilde{A}\|_F^2}
\cdot 
\frac{\lvert \tilde{a}_{i}^Tx_k-\tilde{b}_{i}\rvert^2}{\|\tilde{a}_i\|_2^2}
=
 \frac{\|\tilde{A}x_k-\tilde{b}\|_2^2}{\|\tilde{A}\|_F^2}$$
that $\mu_k$ in (\ref{eq3.5}) is monotonically increasing about $\theta\in [0,1]$ and the candidate set $\mathcal{U}_k$ in (\ref{eq3.3}) is nonempty in each iteration. 
The main difference between RK in \cite{bergou2024note} and Algorithm \ref{al1} is the sampling rule. 
When $\theta=\frac{1}{2}$, the RGRK method reduces to the GRK method in \cite{bai2018greedy}. 

According to Theorem 2.1 in \cite{bai2018relaxed}, we obtain the following proposition concerning $w_k=\frac{\mu_k}{\|\tilde{A}x_k-\tilde{b}\|_2^2}$.
\begin{proposition}[Theorem 2.1 in \cite{bai2018relaxed}]
	\label{prop:1}
	Let $\{w_k\}_{k\in\mathbb{N}}$ be the sequence generated by Algorithm \ref{al1}. Then
	for $k\in\mathbb{N}^{+}$ we have
	\begin{eqnarray}
	\label{eq2:5}
	w_k\geq \frac{1}{\|\tilde{A}\|_F^2}\left(\frac{\theta}{\gamma}\|\tilde{A}\|_F^2+(1-\theta)\right),
	\end{eqnarray}
	where $\gamma=\max_{1\leq i\leq m} \sum_{j=1,j\neq i}^m \|\tilde{a}_{i}\|^2$; and for $k=0$ it holds that
	$$
	w_k\geq \frac{1}{\|\tilde{A}\|_F^2}.
	$$
\end{proposition}

\begin{algorithm}[H]
	\caption{The RGRK method for doubly-noisy linear systems}
	\label{al1}
	\begin{algorithmic}[1]
		\State \textbf{Input}: Given $\tilde{A} \in \mathbb{R}^{m \times n}$, $\tilde{b} \in \mathbb{R}^{m}$, $x_{0}=0\in \mathbb{R}^{n}$, $\theta\in [0,1]$ and $M$
		\State \textbf{Ouput}: $x_M$
		\State \textbf{for} $k=0,1,2, \ldots,M$
		\State $\quad$$\quad$compute \begin{equation}
		\label{eq3.5}
		\mu_k=
	\theta\max_{1\leq i\leq m}\left\{
		\frac{\lvert \tilde{a}_{i}^Tx_k-\tilde{b}_{i}\rvert^2}{\|\tilde{a}_i\|_2^2}
		\right\}
		+(1-\theta)\frac{\|\tilde{A}x_k-\tilde{b}\|_2^2}{\|\tilde{A}\|_F^2}
		\end{equation}
		\State $\quad$$\quad$determine the index set of positive integers
		\begin{equation}
		\label{eq3.3}
		\mathcal{U}_k=\left\{i
		\mid 
		\frac{\lvert \tilde{a}_{i}^Tx_k-\tilde{b}_{i}\rvert^2}{\|\tilde{a}_i\|_2^2}
	 \geq \mu_k
		\right\}
		\end{equation}
		\State $\quad$$\quad$compute the $i$-th entry $\tilde{r}_k^{i}$ of the vector $\tilde{r}_k$ according to
		$$
		\tilde{r}_k^{i}=
		\left\{
		\begin{aligned}
	\tilde{a}_{i}^Tx_k-\tilde{b}_{i},&\text{if}~i\in \mathcal{U}_k\\
		0,& \text{otherwise}\\
		\end{aligned}
		\right.
		$$
		\State$\quad$$\quad$select $i_k\in\mathcal{U}_k$ with probability $P(i=i_k)$=$\frac{\lvert \tilde{r}_k^{i}\rvert^2}{\|\tilde{r}_k\|^2_2}$
		\State$\quad$$\quad$set $x_{k+1}=x_k-\frac{	\tilde{a}_{i_k}^Tx_k-\tilde{b}_{i_k}}{\|\tilde{a}_{i_k}\|^2_2}\tilde{a}_{i_k}$
		\State \textbf{end for}
	\end{algorithmic}
\end{algorithm}

We next establish convergence results for Algorithm \ref{al1} on doubly-noisy linear systems. The analysis shows that the iterates are within a fixed distance from the solution, where the bound depends on the norms of the noise and the solution.

\begin{theorem}
\label{th1}
Assume that the underlying system $Ax=b$ has a solution $\hat{x}=A^{\dagger}b$.
Let $\{x_k\}_{k\in\mathbb{N}}$ be generated by Algorithm \ref{al1} applied to the doubly-noisy system $\tilde{A}x\approx\tilde{b}$ with 
$\tilde{A}=A+E$ and $\tilde{b}=b+\varepsilon$.
 Then starting from any initial point $x_0\in\mathbb{R}^n$ it holds that
 $$
 \mathbb{E}\|	x_1-\hat{x}\|^2
 \leq 
 \left(1-\frac{\tilde{\sigma}_{\min}^2(\tilde{A})}{\|\tilde{A}\|_F^2}\right)
 \mathbb{E}\|x_0-\hat{x}\|_2^2
 +
 \frac{\|E\hat{x}-\varepsilon\|_2^2}{ (\|\tilde{A}\|_F-\frac{\gamma}{\|\tilde{A}\|_F})^2},
 $$ 
 and
\begin{equation*}
\mathbb{E}\|	x_{k+1}-\hat{x}\|^2
\leq 
\left[1-\left(\frac{\theta}{\gamma}\|\tilde{A}\|_F^2+(1-\theta)\right)\cdot\frac{\tilde{\sigma}_{\min}^2(\tilde{A})}{\|\tilde{A}\|_F^2}\right]
\mathbb{E}\|x_k-\hat{x}\|_2^2
+
\frac{\|E\hat{x}-\varepsilon\|_2^2}{ (\|\tilde{A}\|_F-\frac{\gamma}{\|\tilde{A}\|_F})^2},
~k\in\mathbb{N}^{+},
\end{equation*}
where $\gamma=\max_{1\leq i\leq m} \sum_{j=1,j\neq i}^m \|\tilde{a}_{i}\|^2$ and the relaxed parameter $\theta\in [0,1]$.
\end{theorem}

\begin{proof}
Starting with the update in Algorithm \ref{al1}, we have
$$
\begin{aligned}
x_{k+1}-\hat{x}
&=x_k-\hat{x}-\frac{	\tilde{a}_{i_k}^Tx_k-\tilde{b}_{i_k}}{\|\tilde{a}_{i_k}\|^2_2}\tilde{a}_{i_k}\\
&=x_k-\hat{x}-\frac{	\tilde{a}_{i_k}^T(x_k-\hat{x})+E_{i_k}\hat{x}-\varepsilon_{i_k}}{\|\tilde{a}_{i_k}\|^2_2}\tilde{a}_{i_k}.
\end{aligned}
$$	
Hence, we obtain
\begin{equation}
\begin{aligned}
\label{eq3:1}
\|	x_{k+1}-\hat{x}\|^2
=\|x_k-\hat{x}\|^2-\frac{(\tilde{a}_{i_k}^T(x_k-\hat{x}))^2}{\|\tilde{a}_{i_k}\|^2_2}
+\frac{(E_{i_k}\hat{x}-\varepsilon_{i_k})^2}{\|\tilde{a}_{i_k}\|^2_2}.
\end{aligned}
\end{equation}
We fix the values of the indices $i_0,\cdots,i_{k-1}$ and consider only $i_k$ as a random variable. Taking the expectation on (\ref{eq3:1}) conditional to the choices of $i_0,\cdots,i_{k-1}$ obtains
\begin{align}
\label{eq2:2}
\mathbb{E}_k\|	x_{k+1}-\hat{x}\|^2
&=\|x_k-\hat{x}\|^2-\mathbb{E}_k\frac{(\tilde{a}_{i_k}^T(x_k-\hat{x}))^2}{\|\tilde{a}_{i_k}\|^2_2}
+\mathbb{E}_k\frac{(E_{i_k}\hat{x}-\varepsilon_{i_k})^2}{\|\tilde{a}_{i_k}\|^2_2}.
\end{align}

For the second term in (\ref{eq2:2}), we have
\begin{equation}
\begin{aligned}
\label{eq2:3}
\mathbb{E}_k\frac{(\tilde{a}_{i_k}^T(x_k-\hat{x}))^2}{\|\tilde{a}_{i_k}\|^2_2}
&=\sum_{i\in\mathcal{U}_k} \frac{(\tilde{a}_{i}^Tx_k-\tilde{b}_{i})^2}{\sum_{i\in\mathcal{U}_k}(\tilde{a}_{i}^Tx_k-\tilde{b}_{i})^2}\cdot\frac{(\tilde{a}_{i}^T(x_k-\hat{x}))^2}{\|\tilde{a}_{i}\|_2^2}\\
&\geq
w_k \|\tilde{A}x_k-\tilde{b}\|_2^2\cdot
 \frac{\sum_{i\in\mathcal{U}_k}(\tilde{a}_{i}^T(x_k-\hat{x}))^2}{\sum_{i\in\mathcal{U}_k}(\tilde{a}_{i}^Tx_k-\tilde{b}_{i})^2}
\\
&=  
w_k
\|\tilde{A}_{\mathcal{U}_k}(x_k-\hat{x})\|_2^2 
\cdot
\frac{\|\tilde{A}x_k-\tilde{b}\|_2^2}{\|\tilde{A}_{\mathcal{U}_k}x_k-\tilde{b}_{\mathcal{U}_k}\|_2^2}\\
&\geq w_k \tilde{\sigma}_{\min}^2(\tilde{A})\|x_k-\hat{x}\|_2^2.
\end{aligned}
\end{equation}
Here the first inequality follows from the fact that for any $i\in \mathcal{U}_k$ it holds that
\begin{equation}
\label{eq2:4}
 (\tilde{a}_{i}^Tx_k-\tilde{b}_{i})^2 \geq \mu_k \|\tilde{a}_{i}\|_2^2 =w_k \|\tilde{A}x_k-\tilde{b}\|_2^2
 \|\tilde{a}_{i}\|_2^2.
\end{equation}
The last inequality uses $\sigma_{\min}^2(\tilde{A}_{\mathcal{U}_k})\geq \tilde{\sigma}_{\min}^2(\tilde{A})$ together with $$\|\tilde{A}x_k-\tilde{b}\|_2^2\geq \|\tilde{A}_{\mathcal{U}_k}x_k-\tilde{b}_{\mathcal{U}_k}\|_2^2.$$ 
Combining (\ref{eq2:5}) and (\ref{eq2:3}), we deduce that, for $k\in\mathbb{N}^{+}$,
\begin{equation}
\label{eq2:6}
\mathbb{E}_k\frac{(\tilde{a}_{i_k}^T(x_k-\hat{x}))^2}{\|\tilde{a}_{i_k}\|^2_2}\geq 
\left(\frac{\theta}{\gamma}\|\tilde{A}\|_F^2+(1-\theta)\right)\cdot\frac{\tilde{\sigma}_{\min}^2(\tilde{A})}{\|\tilde{A}\|_F^2}\|x_k-\hat{x}\|_2^2,
\end{equation}
and for $k=0$ we have
\begin{equation}
\label{eq2.1}
\mathbb{E}_k\frac{(\tilde{a}_{i_k}^T(x_k-\hat{x}))^2}{\|\tilde{a}_{i_k}\|^2_2}\geq 
\frac{\tilde{\sigma}_{\min}^2(\tilde{A})}{\|\tilde{A}\|_F^2}\|x_k-\hat{x}\|_2^2.
\end{equation}

For the third term in (\ref{eq2:2}), we have
\begin{equation}
\begin{aligned}
\label{eq2:7}
\mathbb{E}_k\frac{(E_{i_k}\hat{x}-\varepsilon_{i_k})^2}{\|\tilde{a}_{i_k}\|^2_2}
&
=\sum_{i\in\mathcal{U}_k} \frac{(\tilde{a}_{i}^Tx_k-\tilde{b}_{i})^2}{\sum_{i\in\mathcal{U}_k}(\tilde{a}_{i}^Tx_k-\tilde{b}_{i})^2}\cdot\frac{(E_i\hat{x}-\varepsilon_{i_k})^2}{\|\tilde{a}_{i}\|^2_2}.
\end{aligned}
\end{equation}
It follows from (\ref{eq2:4}) that
we have 
\begin{equation}
\begin{aligned}
\label{eq2:8}
\sum_{i\in\mathcal{U}_k}(\tilde{a}_{i}^Tx_k-\tilde{b}_{i})^2\geq
\mu_k \sum_{i\in\mathcal{U}_k}\|\tilde{a}_{i}\|_2^2\geq \mu_k \min_{1\leq i\leq m}\|\tilde{a}_{i}\|_2^2.
\end{aligned}
\end{equation}
Inserting (\ref{eq2:8}) into (\ref{eq2:7}), it holds that
\begin{equation}
\begin{aligned}
\label{eq2:9}
\mathbb{E}_k\frac{(E_{i_k}\hat{x}-\varepsilon_{i_k})^2}{\|\tilde{a}_{i_k}\|^2_2}
&\leq 
\frac{1}{\mu_k \min_{1\leq i\leq m}\|\tilde{a}_{i}\|^2_2}\sum_{i\in\mathcal{U}_k} \frac{(\tilde{a}_{i}^Tx_k-\tilde{b}_{i})^2}{\|\tilde{a}_{i}\|^2_2}\cdot (E_{i}\hat{x}-\varepsilon_{i})^2\\
&\leq 
\frac{1}{\mu_k }\cdot \frac{\|\tilde{A}x_k-\tilde{b}\|_2^2}{(\min_{1\leq i\leq m}\|\tilde{a}_{i}\|^2_2)^2}\cdot
\|E_{\mathcal{U}_k}\hat{x}-\varepsilon_{\mathcal{U}_k}\|_2^2\\
&\leq 
\frac{\|E\hat{x}-\varepsilon\|_2^2}{ (\|\tilde{A}\|_F-\frac{\gamma}{\|\tilde{A}\|_F})^2},
\end{aligned}
\end{equation}
where the last inequality follows from $\mu_k \geq \frac{\|\tilde{A}x_k-\tilde{b}\|_2^2}{\|\tilde{A}\|_F^2}$ and $\|E_{\mathcal{U}_k}\hat{x}-\varepsilon_{\mathcal{U}_k}\|_2^2\leq \|E\hat{x}-\varepsilon\|_2^2$.

Finally, combining (\ref{eq2:2}), (\ref{eq2:6}), (\ref{eq2.1}), and (\ref{eq2:9}), taking the full expectation on both sides yields
$$
\mathbb{E}\|	x_1-\hat{x}\|^2
\leq 
\left(1-\frac{\tilde{\sigma}_{\min}^2(\tilde{A})}{\|\tilde{A}\|_F^2}\right)
\mathbb{E}\|x_0-\hat{x}\|_2^2
+
\frac{\|E\hat{x}-\varepsilon\|_2^2}{ (\|\tilde{A}\|_F-\frac{\gamma}{\|\tilde{A}\|_F})^2},
$$ 
and
\begin{equation*}
\mathbb{E}\|	x_{k+1}-\hat{x}\|^2
\leq 
\left[1-\left(\frac{\theta}{\gamma}\|\tilde{A}\|_F^2+(1-\theta)\right)\cdot\frac{\tilde{\sigma}_{\min}^2(\tilde{A})}{\|\tilde{A}\|_F^2}\right]
\mathbb{E}\|x_k-\hat{x}\|_2^2
+
\frac{\|E\hat{x}-\varepsilon\|_2^2}{ (\|\tilde{A}\|_F-\frac{\gamma}{\|\tilde{A}\|_F})^2},
~k\in\mathbb{N}^{+}.
\end{equation*}
The proof is completed.
\end{proof}

\begin{remark}
	\label{remark:1}
	\begin{enumerate}[(i)]
		\item In view of $\|\tilde{A}\|_F^2>\gamma$, an upper bound of convergence factor of RGRK is monotonically decreasing with respect to $\theta\in [0,1]$. When $\theta=1$, the relaxed greedy sampling method reduces to the greedy sampling method, also known as maximal correction method in \cite{feichtinger1992new}.
		\item As $\frac{\theta}{\gamma}\|\tilde{A}\|_F^2+(1-\theta)>1$, it holds that$$1-\left(\frac{\theta}{\gamma}\|\tilde{A}\|_F^2+(1-\theta)\right)\cdot\frac{\tilde{\sigma}_{\min}^2(\tilde{A})}{\|\tilde{A}\|_F^2}<1-\frac{\tilde{\sigma}_{\min}^2(\tilde{A})}{\|\tilde{A}\|_F^2}.$$
		The convergence horizon $\frac{\|E\hat{x}-\varepsilon\|_2^2}{\|\tilde{A}\|_F^2}$ of Theorem 3.1 in \cite{bergou2024note} is changed to 
		$\frac{\|E\hat{x}-\varepsilon\|_2^2}{ (\|\tilde{A}\|_F-\frac{\gamma}{\|\tilde{A}\|_F})^2}$, which remains to be solved in the next section.
		\item When there is no noise in both coefficient matrix and measurement vector, Theorem \ref{th2} reduces to
		$$\mathbb{E}\|	x_{k+1}-\hat{x}\|^2
		\leq 
		\left[1-\left(\frac{\theta}{\gamma}\|A\|_F^2+(1-\theta)\right)\cdot\frac{\sigma_{\min}^2(A)}{\|A\|_F^2}\right]
		\mathbb{E}\|x_k-\hat{x}\|_2^2,
		~k\in\mathbb{N}^{+},$$
		which coincides with Theorem 2.1 in \cite{bai2018relaxed}. 
	\end{enumerate}
\end{remark}
	
\subsection{Multiplicative Noise}

We have analyzed the additive doubly-noisy systems, and now extend to the more general multiplicative doubly-noisy case. Specifically, we consider the following linear system
\begin{equation}
\label{eq4:1}
(I_m+E)A(I_n+F)x\approx b+\varepsilon,
\end{equation}
where $I_m+E$ and $I_n+F$ are nonsingular, and $A\in\mathbb{R}^{m\times n}$ and $b\in\mathbb{R}^{m}$ represent the exact data. Let $\triangle A=EA+AF+EAF$, then we have $$\tilde{A}=(I_m+E)A(I_n+F)=A+\triangle A,~\tilde{b}=b+\varepsilon.$$
Applying the analysis of Theorem \ref{th1} to the multiplicative doubly-noisy system (\ref{eq4:1}), we obtain the following convergence result.

\begin{lemma}
	\label{th3}
	Assume that the underlying system $Ax=b$ has a solution $\hat{x}=A^{\dagger}b$.
	Let $\{x_k\}_{k\in\mathbb{N}}$ be generated by Algorithm \ref{al1} applied to the doubly-noisy system $\tilde{A}x\approx\tilde{b}$ with 
nonsingular $\tilde{A}=(I_m+E)A(I_n+F)=A+\triangle A$ and $\tilde{b}=b+\varepsilon$.
	Then starting from any initial point $x_0\in\mathbb{R}^n$ it holds that
	\begin{equation*}
	\mathbb{E}\|x_{k+1}-\hat{x}\|^2
	\leq 
	\left[1-\left(\frac{\theta}{\gamma}\|\tilde{A}\|_F^2+(1-\theta)\right)\cdot\frac{\tilde{\sigma}_{\min}^2(\tilde{A})}{\|\tilde{A}\|_F^2}\right]
	\mathbb{E}\|x_k-\hat{x}\|_2^2
	+
	\frac{\|\triangle A\hat{x}-\varepsilon\|_2^2}{ (\|\tilde{A}\|_F-\frac{\gamma}{\|\tilde{A}\|_F})^2},
	~k\in\mathbb{N}^{+},
	\end{equation*}
	where $\gamma=\max_{1\leq i\leq m} \sum_{j=1,j\neq i}^m \|\tilde{a}_{i}\|^2$ and the relaxed parameter $\theta\in [0,1]$.
\end{lemma}

\section{The RGRK method with signal averaging}
\label{sec3}

The convergence horizon in Theorem \ref{th1} inevitably limits the performance of the RGRK method on doubly-noisy systems (\ref{eq3:2}). A natural question is how to obtain a more accurate approximated solution to doubly-noisy systems. As discussed in Section \ref{sub2.2}, the signal averaging technique can reduce the noise and enhance the true signal by using the average of multiple measurements.
To the best of our knowledge, the signal averaging technique has not been studied for RK or its variants in the setting of doubly-noisy systems.

Assume that the measurement data $\tilde{A}$ and $\tilde{b}$ can be measured multiple times repeatedly, while the solution to the underlying system $Ax=b$ does not change during measurements.
In detail, we are only given unbiased, independent, identically distributed measurements $A^1,\ldots,A^N$ and $b^1,\ldots,b^N$. 
Let the $i$-th row of $A^j$ be $a_i^j$ and the $i$-th element of $b^j$ be $b_i^j$ for any $i=1,\ldots,m,j=1,\ldots,N$.
According to $A^j=A+E^j$ and $b^j=b+\varepsilon^j$ for any $j=1,\ldots,N$, we have $a_i^j=a_i+E_i^j$ and $b_i^j=b_i+\varepsilon_i^j$.
To apply the signal averaging technique, we make the following assumptions.
\begin{assumption}
	\label{assum1}
\begin{enumerate}[(i)]
	\item The exact system $Ax=b$ is consistent with $\hat{x}=A^{\dagger}b$.
\item Each row of $A^j$ is normalized, i.e., $\|a_i^j\|_2=1$ for all $i=1,\ldots,m,j=1,\ldots,N$.
\item The measurement noise $E_i^j$ and $\varepsilon_i^j$ satisfy
$$\mathbb{E}E_i^j=0,~\mathbb{E}\|E_i^j\|^2=\sigma_{E,i}^2,~\mathbb{E}\varepsilon_i^j=0,~\mathbb{E}(\varepsilon_i^j)^2=\sigma_{\varepsilon,i}^2,$$
for all $i=1,\ldots,m$ and $j=1,\ldots,N$. 
\end{enumerate}
\end{assumption} 

For all $j=1,\ldots,N$, denote the total noise level of $E^j$ and $\varepsilon^j$ by 
$$\sum_{i=1}^{m} \sigma_{E,i}^2= \sigma_{E}^2,~\sum_{i=1}^{m} \sigma_{\varepsilon,i}^2=\sigma_{\varepsilon}^2.$$
In the $k$-th iterate, we sample $i_k$-th hyperplane and utilize the averages
\begin{equation}
\label{eq3.6}
\bar{a}_{i_k}=\frac{1}{N}\sum_{j=1}^{N} a_{i_k}^j,~\bar{b}_{i_k}=\frac{1}{N}\sum_{j=1}^{N} b_{i_k}^j
\end{equation}
as unbasied estimators of $a_{i_k}$ and $b_{i_k}$, instead of $\tilde{a}_{i_k}$ and $\tilde{b}_{i_k}$.
It follows from (\ref{eq3.6}) that we have
\begin{equation}
\label{eq3.1}
\mathbb{E} \|\bar{a}_{i_k}-a_{i_k}\|^2
=\mathbb{E} \|\bar{E}_{i_k}\|^2
=\mathbb{E} \left\|\frac{1}{N}\sum_{j=1}^{N} E_{i_k}^j\right\|^2
=\frac{1}{N^2}\sum_{j=1}^{N}\mathbb{E}\|E_{i_k}^j\|^2
=\frac{\sigma_{E,i_k}^2}{N},
\end{equation}
and
\begin{equation}
\label{eq3.2}
\mathbb{E} \|\bar{b}_{i_k}-b_{i_k}\|^2
=\mathbb{E} \|\bar{\varepsilon}_{i_k}\|^2
=\mathbb{E} \left\|\frac{1}{N}\sum_{j=1}^{N} \varepsilon_{i_k}^j\right\|^2
=\frac{1}{N^2}\sum_{j=1}^{N}\mathbb{E}(\varepsilon_{i_k}^j)^2
=\frac{\sigma_{\varepsilon,i_k}^2}{N},
\end{equation}
which demonstrates that the signal averaging technique reduces the variance by a factor of $N$.
From a statistical perspective, the noise reduction is natural since the variance of the averages $\bar{a}_{i_k}$ and $\bar{b}_{i_k}$
quantifies the deviation of the estimators from the true values $a_{i_k}$ and $b_{i_k}$, respectively.
Consequently, employing the averaged data is expected to achieve a more accurate approximate solution than relying on a single measurement.

Motivated by this observation, we integrate the RGRK method with the signal averaging technique and propose the RGRK with signal averaging (RGRK-SA) method for doubly-noisy linear systems, summarized in Algorithm \ref{al2}. Since the averaged matrix $\bar{A}$ and averaged vector $\bar{b}$ can be precomputed, Algorithm \ref{al1} shares the same computational cost as Algorithm \ref{al2} in each iterate. This modification makes the RGRK method more robust and accurate in doubly-noisy cases.

\begin{algorithm}[h]
	\caption{The RGRK-SA method for doubly-noisy linear systems}
	\label{al2}
	\begin{algorithmic}[1]
		\State \textbf{Input}: Given $A^1,\ldots,A^j \in \mathbb{R}^{m \times n}$, $b^1,\ldots,b^N \in \mathbb{R}^{m}$, $x_{0}=0\in \mathbb{R}^{n}$, $\theta\in [0,1]$ and $M$
		\State \textbf{Ouput}: $x_M$
		\State compute $\bar{a}_i=\frac{1}{N}\sum_{j=1}^{N}a_i^j,~\bar{b}_i=\frac{1}{N}\sum_{j=1}^{N}b_i^j,~i=1,\ldots,m$
		\State \textbf{for} $k=0,1,2, \ldots,M$
		\State $\quad$$\quad$compute \begin{equation*}
		\mu_k=
		\theta\max_{1\leq i\leq m}\left\{
		\frac{\lvert \bar{a}_{i}^Tx_k-\bar{b}_{i}\rvert^2}{\|\bar{a}_i\|_2^2}
		\right\}
		+(1-\theta)\frac{\|\bar{A}x_k-\bar{b}\|_2^2}{\|\bar{A}\|_F^2}
		\end{equation*}
		\State $\quad$$\quad$determine the index set of positive integers
		\begin{equation*}
		\mathcal{U}_k=\left\{i
		\mid 
		\frac{\lvert \bar{a}_{i}^Tx_k-\bar{b}_{i}\rvert^2}{\|\bar{a}_i\|_2^2}
		\geq \mu_k
		\right\}
		\end{equation*}
		\State $\quad$$\quad$compute the $i$-th entry $\tilde{r}_k^{i}$ of the vector $\tilde{r}_k$ according to
		$$
		\tilde{r}_k^{i}=
		\left\{
		\begin{aligned}
		\bar{a}_{i}^Tx_k-\bar{b}_{i},&\text{if}~i\in \mathcal{U}_k\\
		0,& \text{otherwise}\\
		\end{aligned}
		\right.
		$$
		\State$\quad$$\quad$select $i_k\in\mathcal{U}_k$ with probability $P(i=i_k)$=$\frac{\lvert \bar{r}_k^{i}\rvert^2}{\|\bar{r}_k\|^2_2}$
		\State$\quad$$\quad$set $x_{k+1}=x_k-	\frac{\bar{a}_{i_k}^Tx_k-\bar{b}_{i_k}}{\|\bar{a}_{i_k}\|_2^2}\bar{a}_{i_k}$
		\State \textbf{end for}
	\end{algorithmic}
\end{algorithm}

\begin{remark}
	\label{remark:4}
\begin{enumerate}[(i)]
	\item When $N=1$, Algorithm \ref{al2} reduces to Algorithm \ref{al1}.
	\item As $N$ increases, it follows from (\ref{eq3.1}) and (\ref{eq3.2}) that the variance of the averages $\bar{a}_{i_k}$ and $\bar{b}_{i_k}$
	decreases, leading to more precise and unbiased estimators of the true signals $a_{i_k}$ and $b_{i_k}$.  
\end{enumerate}
\end{remark}

Building on the signal averaging technique, we can now present an improved convergence result for Algorithm \ref{al2}.

\begin{theorem}
\label{th2}
Let $\{x_k\}_{k\in\mathbb{N}}$ be generated by Algorithm \ref{al2} applied to the doubly-noisy system $\tilde{A}x\approx\tilde{b}$ with 
$\tilde{A}=A+E$ and $\tilde{b}=b+\varepsilon$.
Under Assumption \ref{assum1}, 
starting from any initial point $x_0\in\mathbb{R}^n$ it holds that
$$
\mathbb{E}\|x_{1}-\hat{x}\|^2
\leq
\left(
1- \frac{\mathbb{E}\tilde{\sigma}_{\min}^2(\bar{A})}{m}
\right)\mathbb{E}\|x_0-\hat{x}\|^2
+
\frac{1}{N}\cdot\frac{\sigma_{E}^2\|\hat{x}\|^2+\sigma_{\varepsilon}^2
}{ (\sqrt{m}-\frac{\gamma}{\sqrt{m}})^2},
$$
and
\begin{equation*}
\mathbb{E}\|x_{k+1}-\hat{x}\|^2
\leq
\left[
1-\left(\frac{\theta}{\gamma}m+(1-\theta)\right) \frac{\mathbb{E}\tilde{\sigma}_{\min}^2(\bar{A})}{m}
\right]\mathbb{E}\|x_k-\hat{x}\|^2
+
\frac{1}{N}\cdot\frac{\sigma_{E}^2\|\hat{x}\|^2+\sigma_{\varepsilon}^2
}{ (\sqrt{m}-\frac{\gamma}{\sqrt{m}})^2},~k\in \mathbb{N}^+,
\end{equation*}
where $\gamma=m-1$ and the relaxed parameter $\theta\in [0,1]$.
\end{theorem}

\begin{proof}
Borrowing a proof similar to that of Theorem \ref{th1}, we start with
$$
\begin{aligned}
\|	x_{k+1}-\hat{x}\|^2
=\|x_k-\hat{x}\|^2-(\bar{a}_{i_k}^T(x_k-\hat{x}))^2
+(\bar{E}_{i_k}\hat{x}-\bar{\varepsilon}_{i_k})^2.
\end{aligned}
$$
Taking expectation on $i_k,\bar{E}$ and $\bar{\varepsilon}$ conditioning with $x_k$ obtains
\begin{equation}
\label{eq3.4}
\mathbb{E}_{\bar{E},\bar{\varepsilon}}\mathbb{E}_k\|x_{k+1}-\hat{x}\|^2
=\|x_k-\hat{x}\|^2-\mathbb{E}_{\bar{E}}\mathbb{E}_{k}(\bar{a}_{i_k}^T(x_k-\hat{x}))^2
+\mathbb{E}_{\bar{E},\bar{\varepsilon}}\mathbb{E}_k(\bar{E}_{i_k}\hat{x}-\bar{\varepsilon}_{i_k})^2.
\end{equation}

It follows from the proof of Theorem \ref{th1} that we have
\begin{equation}
\label{eq3.7}
\mathbb{E}_{\bar{E}}\mathbb{E}_k(\bar{a}_{i_k}^Tx_k-\bar{b}_{i_k})^2
\geq
\frac{\mathbb{E}_{\bar{E}}\tilde{\sigma}_{\min}^2(\bar{A})}{m}
\|x_k-\hat{x}\|_2^2,~k=0,
\end{equation}
\begin{equation}
\label{eq3.9}
\mathbb{E}_{\bar{E}}\mathbb{E}_k(\bar{a}_{i_k}^Tx_k-\bar{b}_{i_k})^2
\geq
\left(\frac{\theta}{\gamma}m+(1-\theta)\right)\cdot\frac{\mathbb{E}_{\bar{E}}\tilde{\sigma}_{\min}^2(\bar{A})}{m}
\|x_k-\hat{x}\|_2^2,~k\in\mathbb{N}^{+},
\end{equation}
and
\begin{equation}
\begin{aligned}
\label{eq3.8}
\mathbb{E}_{\bar{E},\bar{\varepsilon}}\mathbb{E}_k
\|\bar{E}_{i_k}\hat{x}-\varepsilon_{i_k}\|^2
&\leq
\frac{\mathbb{E}_{\bar{E},\bar{\varepsilon}}\|\bar{E}\hat{x}-\varepsilon\|_2^2}{ (\sqrt{m}-\frac{\gamma}{\sqrt{m}})^2}
\\
&=\frac{\mathbb{E}_{\bar{E}}\|\bar{E}\hat{x}\|^2+\mathbb{E}_{\bar{\varepsilon}}\|\bar{\varepsilon}\|_2^2}{ (\sqrt{m}-\frac{\gamma}{\sqrt{m}})^2}\\
&
\leq 
\frac{\mathbb{E}_{\bar{E}}\|\bar{E}\|^2\cdot\|\hat{x}\|^2+\mathbb{E}_{\bar{\varepsilon}}\|\bar{\varepsilon}\|_2^2}{ (\sqrt{m}-\frac{\gamma}{\sqrt{m}})^2}\\
&=
\frac{\sigma_{E}^2\|\hat{x}\|^2+\sigma_{\varepsilon}^2
	}{ N(\sqrt{m}-\frac{\gamma}{\sqrt{m}})^2}.
\end{aligned}
\end{equation}
Combining (\ref{eq3.4}), (\ref{eq3.7}), (\ref{eq3.9}), and (\ref{eq3.8}), we have
$$
\mathbb{E}_{\bar{E},\bar{\varepsilon}}\mathbb{E}_k\|x_{k+1}-\hat{x}\|^2
\leq
\left(1-\frac{\mathbb{E}_{\bar{E}}\tilde{\sigma}_{\min}^2(\bar{A})}{m}\right)\|x_k-\hat{x}\|^2
+\frac{\sigma_{E}^2\|\hat{x}\|^2+\sigma_{\varepsilon}^2
}{ N(\sqrt{m}-\frac{\gamma}{\sqrt{m}})^2},~k=0,
$$
and
$$
\mathbb{E}_{\bar{E},\bar{\varepsilon}}\mathbb{E}_k\|x_{k+1}-\hat{x}\|^2
\leq
\left[1-\left(\frac{\theta}{\gamma}m+(1-\theta)\right)\cdot\frac{\mathbb{E}_{\bar{E}}\tilde{\sigma}_{\min}^2(\bar{A})}{m}\right]\|x_k-\hat{x}\|^2
+\frac{\sigma_{E}^2\|\hat{x}\|^2+\sigma_{\varepsilon}^2
}{ N(\sqrt{m}-\frac{\gamma}{\sqrt{m}})^2},~k\in\mathbb{N}^{+}.
$$
Hence, taking the full expectation on both sides, the proof is completed.
\end{proof}

Theorem \ref{th2} decomposes the residual into
the sum of a bias term that decays exponentially fast with the number of iterations $k$ and a variance term that decays at the rate $1/N$ with the number of measurements $N$.
It leads to the linear convergence
of the iterates at the beginning and an $O(1/N)$ rate of convergence later on.

\begin{remark}
\label{remark:2}
It follows from Theorem \ref{th2} that the convergence rate of RGRK-SA increases monotonically with the value of the parameter $\theta$.
\end{remark}

\begin{remark}
	\label{remark:3}
Under Assumption \ref{assum1} with $N=1$, the convergence result in Theorem \ref{th1} can be reformulated as 
$$
\mathbb{E}\|x_{k+1}-\hat{x}\|^2
\leq 
\left[1-\left(\frac{\theta}{\gamma}m+(1-\theta)\right)\cdot\frac{\mathbb{E}\tilde{\sigma}_{\min}^2(\tilde{A})}{m}\right]
\mathbb{E}\|x_k-\hat{x}\|_2^2
+
\frac{\sigma_{E}^2\|\hat{x}\|^2+\sigma_{\varepsilon}^2
}{ (\sqrt{m}-\frac{\gamma}{\sqrt{m}})^2},~k\in\mathbb{N}^{+}.
$$
By employing $N$ noisy measurements and averaging them, our proposed RGRK-SA reduces the radius of the ball from $\frac{\sqrt{\sigma_{E}^2\|\hat{x}\|^2+\sigma_{\varepsilon}^2}
}{\sqrt{m}-\frac{\gamma}{\sqrt{m}}}$ to 
$\frac{1}{\sqrt{N}}\cdot\frac{\sqrt{\sigma_{E}^2\|\hat{x}\|^2+\sigma_{\varepsilon}^2}
}{\sqrt{m}-\frac{\gamma}{\sqrt{m}}}$.
Consequently, we have $\mathbb{E}\|x_{k}-\hat{x}\|^2\rightarrow 0$ as $k\rightarrow \infty$ and $N\rightarrow \infty$.
To conclude, RGRK equipped with signal averaging can eliminate the finite
convergence horizon caused by noise.
\end{remark}

The corollary 7.3.5 in \cite{horn2012matrix} proposed a pair of singular value perturbation results, which is summarized as follows.
\begin{proposition}[Corollary 7.3.5 in \cite{horn2012matrix}]
\label{prop:2}
Let $A,E\in\mathbb{C}^{m\times n}$ and $r=\min\{m,n\}$. Let $\sigma_1(A)\geq \cdots\geq \sigma_{q}(A)$ and $\sigma_1(B)\geq \cdots\geq \sigma_{q}(B)$ be the nonincreasingly ordered singular values of $A$ and $E$, respectively. Then $|\sigma_i(A)-\sigma_i(B)|\leq \||A-B|\|_2$ for each $i=1,\cdots,q$.
\end{proposition}
In this paper, we have $A,E\in\mathbb{R}^{m\times n}$. It follows from Proposition \ref{prop:2} that we obtain
$$\mathbb{E}_{\bar{E}}\tilde{\sigma}_{\min}^2(\tilde{A})=\mathbb{E}_{\bar{E}}\tilde{\sigma}_{\min}^2(A+\bar{E})\geq \mathbb{E}_{\bar{E}}\left(\max\{\sigma_{\min}(A)-\|\bar{E}\|_2,0\}\right)^2,$$
which can be used to reformulate the convergence rate established in Theorem \ref{th2}.
In particular, by invoking the definition of expectation, we obtain
$$
\begin{aligned}
\mathbb{E}_{\bar{E}}\left(\max\{\sigma_{\min}(A)-\|\bar{E}\|_2,0\}\right)^2
&\geq
\left(\max\{\sigma_{\min}(A)-t,0\}\right)^2P(\|\bar{E}\|_2\leq t)\\
&= 
\left(\max\{\sigma_{\min}(A)-t,0\}\right)^2\left(1-P(\|\bar{E}\|_2^2\geq t^2)\right).
\end{aligned}
$$
Applying Markov's inequality yields
$$
\begin{aligned}
	\mathbb{E}_{\bar{E}}\left(\max\{\sigma_{\min}(A)-\|\bar{E}\|_2,0\}\right)^2
&\geq
\left(\max\{\sigma_{\min}(A)-t,0\}\right)^2\left(1-\frac{\mathbb{E}\|\bar{E}\|_2^2}{t^2}\right)\\
&\geq 
\left(\max\{\sigma_{\min}(A)-t,0\}\right)^2\left(1-\frac{\mathbb{E}\|\bar{E}\|_F^2}{t^2}\right)\\
&=\left(\max\{\sigma_{\min}(A)-t,0\}\right)^2\left(1-\frac{\sigma_{E}^2}{Nt^2}\right),
\end{aligned}
$$
where the second inequality follows from $\|\bar{E}\|_2^2\leq \|\bar{E}\|_F^2$. If we let $t=\sqrt{\frac{2}{N}}\sigma_{E}<\sigma_{\min}(A)$, then we arrive at
$$	\mathbb{E}_{\bar{E}}\tilde{\sigma}_{\min}^2(\tilde{A})\geq
\frac{1}{2}\left(\sigma_{\min}(A)-\sqrt{\frac{2}{N}}\sigma_{E}\right)^2.$$

\section{Numerical Experiments}
\label{sec4}

In this section, we carry out some numerical experiments to compare the performance of RGRK, RGRK-SA, and RK in \cite{bergou2024note} for solving doubly-noisy linear systems with two types of coefficient matrices. 
The first type of exact matrix $A\in\mathbb{R}^{m\times n}$ is the random Gaussian matrix generated by using MATLAB function ‘randn’; the second type is chosen from the SuiteSparse Matrix Collection in \cite{davis2011university}. 
The exact vector $b\in\mathbb{R}^m$ is taken to be $b=A\hat{x}$, where the exact solution $\hat{x}\in\mathbb{R}^n$ is generated by MATLAB function 'randn'. 
For $j=1,\ldots,N$, we define the noisy matrices and vectors by $\tilde{A}^{j}=A+\sigma_{E}E^{j}$ and $\tilde{b}^{j}=b+\sigma_{\varepsilon}\varepsilon^{j}$, where $\sigma_{E},\sigma_{\varepsilon}\geq 0$ denote the noise magnitudes, and the noise $E^{j}\in\mathbb{R}^{m\times n}$ and $\varepsilon^{j}\in\mathbb{R}^{m}$ are generated by MATLAB function 'randn'. In our tests, let $\sigma_{E}=1\%,\sigma_{\varepsilon}=1\%$.

All implementations are initialized with $x_{0}=0\in \mathbb{R}^{n}$ and
terminate either when the maximum number of iterations $M$ is reached or the relative error falls below $10^{-1}$.
The relative error at $k$-th iterate is defined by
$$\text{Error}(k)=\frac{\|x_k-\hat{x}\|}{\|\hat{x}\|}.$$
In each experiment, we record the median of residual error over 50 trials. 
All experiments are performed with MATLAB (version R2021b) on a personal computer with
2.80GHZ CPU(Intel(R) Core(TM) i7-1165G7), 16-GB memory, and Windows operating system
(Windows 10).

\subsection{The effect of $\theta$ and $N$}

We perform experiments on doubly-noisy systems generated by $400\times 200$
random Gaussian matrices $A\in\mathbb{R}^{m\times n}$ to test the effect of $\theta$ and $N$ on RGRK and RGRK-SA, respectively.
The maximum number of iterations is set to $M=4000$.

Fix $N=10$ and select different values of $\theta$ from the set $\{0.2,0.4,0.6,0.8,1\}$. 
Figure \ref{fig4} displays the median residuals over 50 trials for RGRK and RGRK-SA, varying in different values of the parameter $\theta$.
According to Figure \ref{fig4}, we find that both RGRK and RGRK-SA achieve faster convergence speed with increasing values of the parameter $\theta$, which is consistent with the theoretical analysis in Remark \ref{remark:1} and Remark \ref{remark:2}.

Fix $\theta=1$ and vary $N$ over the set $\{1,10,10^2,10^3\}$. Figure \ref{fig3} illustrates the median residuals over 50 trials for RGRK-SA under varying values of $N$. As shown in Figure \ref{fig3}, the accuracy of RGRK-SA in recovering the solution of doubly-noisy systems improves as the number of observations $N$ increases, which validates Remark \ref{remark:3}.

\begin{figure}[H]
	\centering
	\subfigure[RGRK]{
		\includegraphics[width=0.33\linewidth,height=0.27\textwidth]{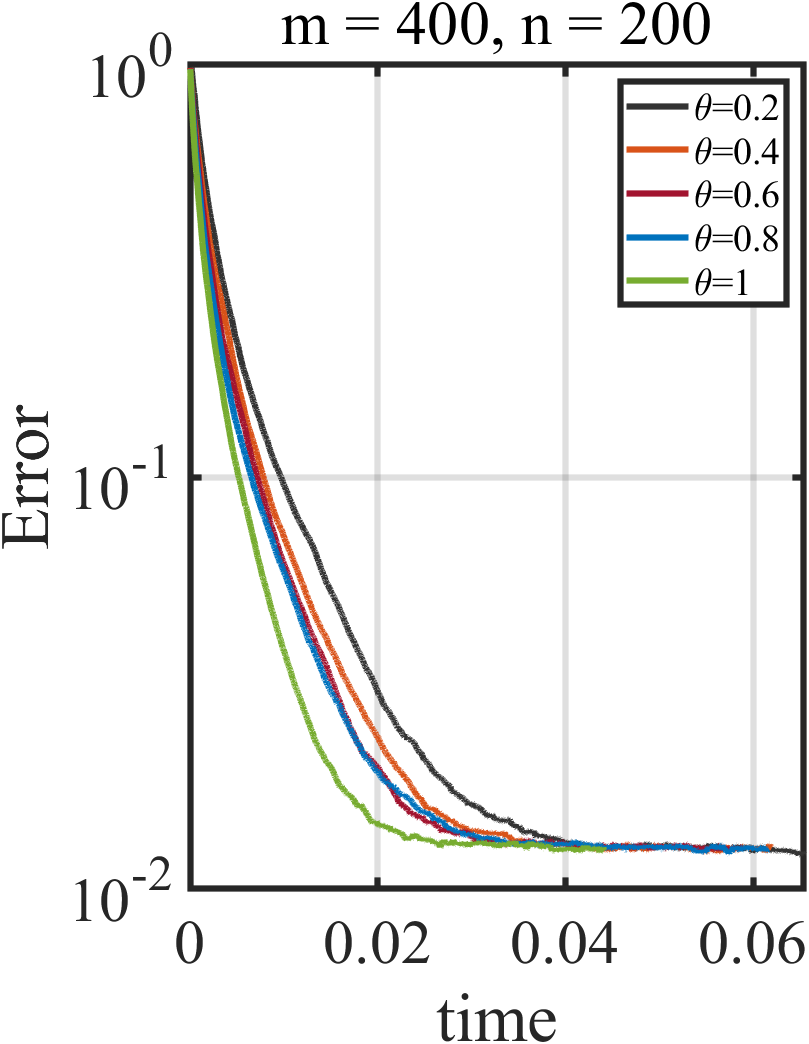}
	}
	\subfigure[RGRK-SA]{
		\includegraphics[width=0.33\linewidth,height=0.27\textwidth]{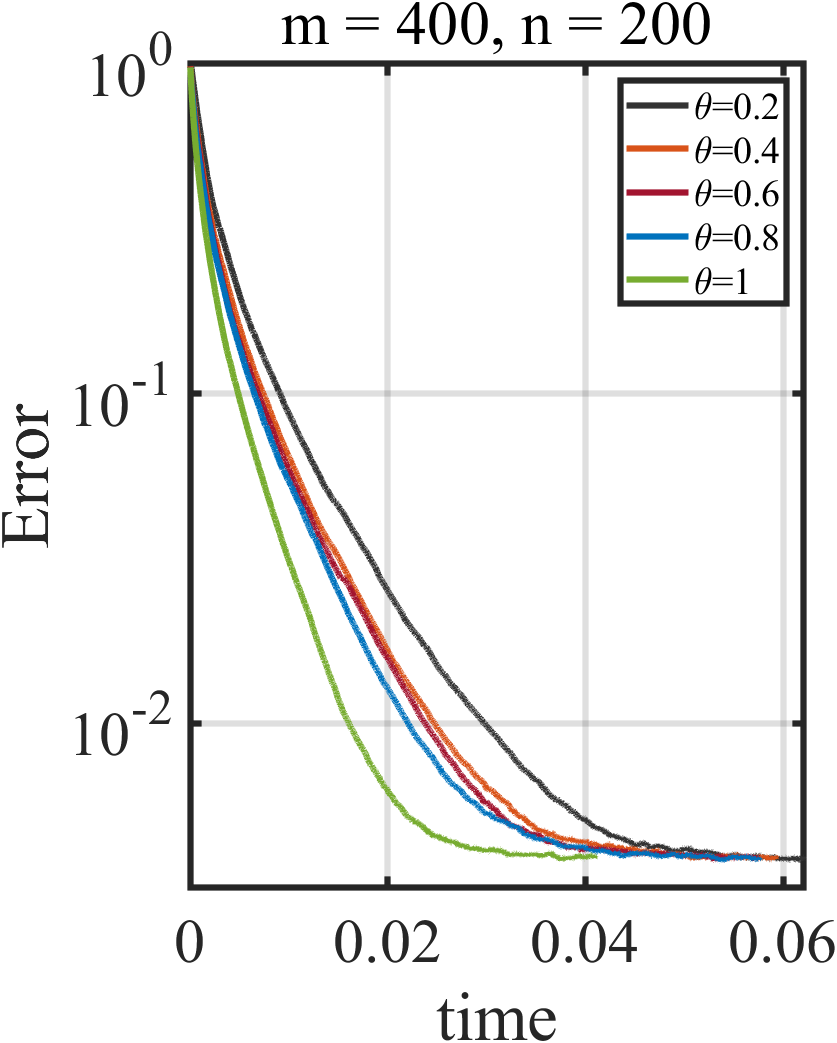}}
	\caption{The effect of $\theta$ on RGRK and RGRK-SA}
	\label{fig4}
\end{figure}

\begin{figure}[H]
	\centering
	\subfigure{
		\includegraphics[width=0.33\linewidth,height=0.27\textwidth]{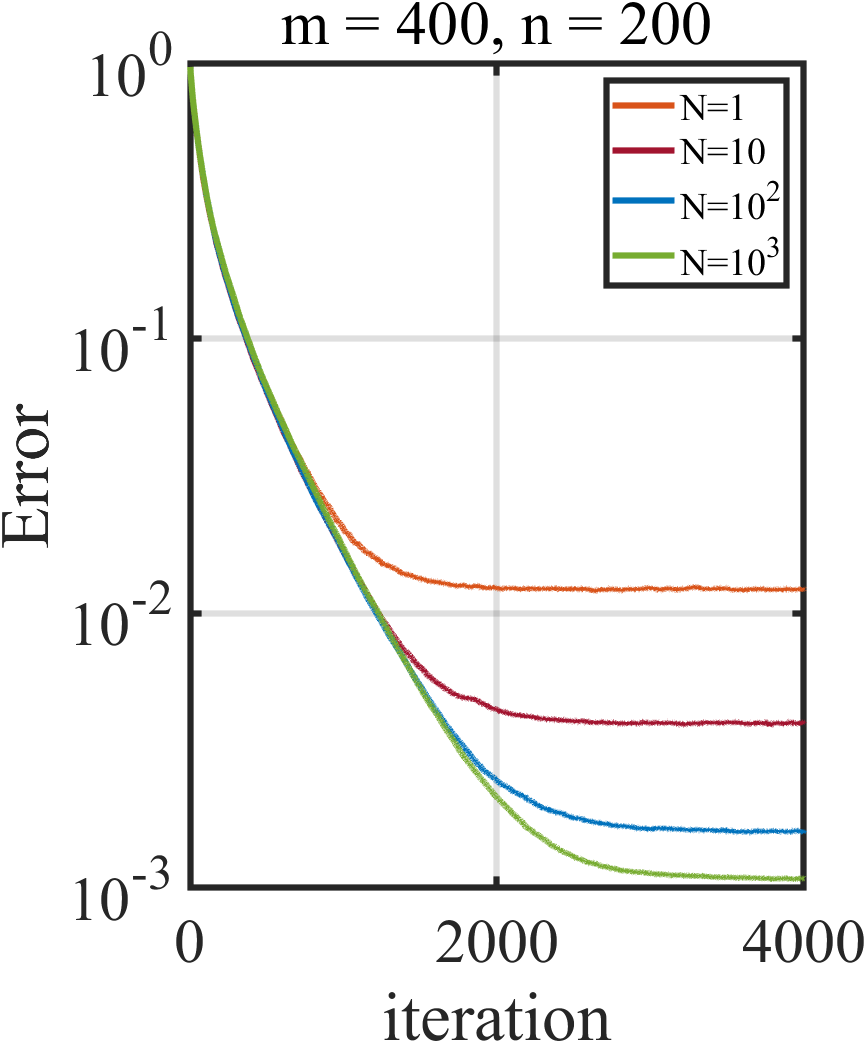}
		\includegraphics[width=0.33\linewidth,height=0.27\textwidth]{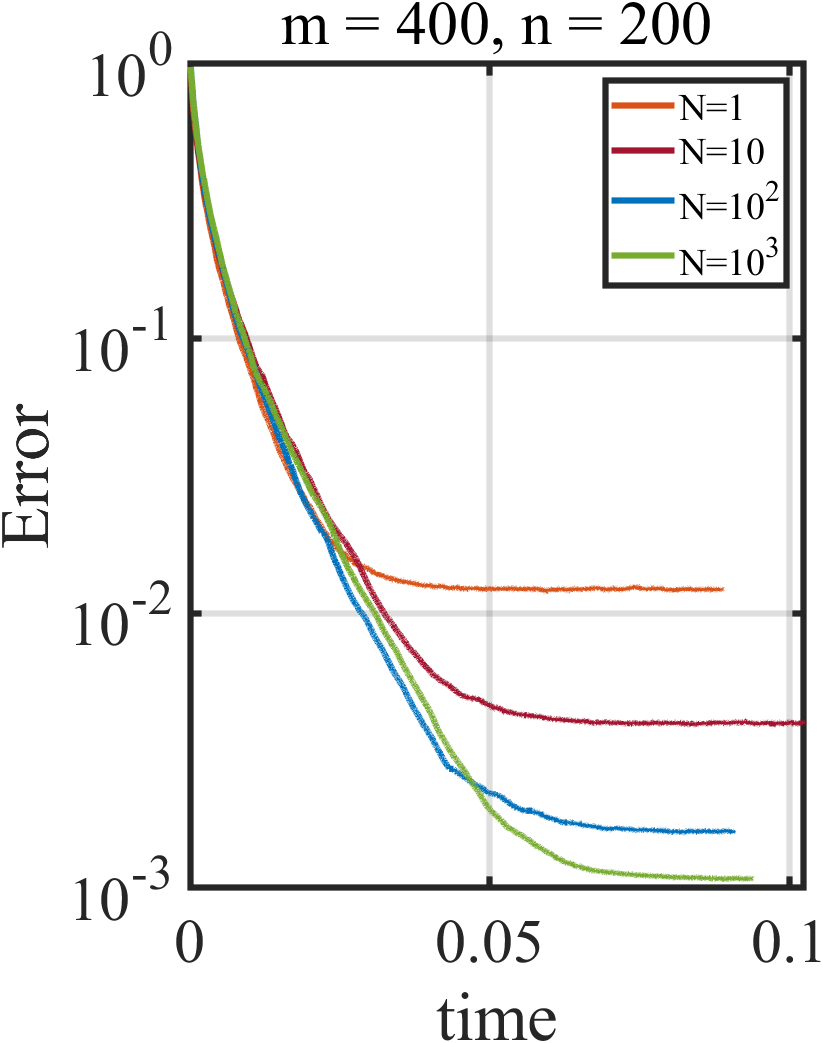}}
	\caption{The effect of $N$ on RGRK-SA}
	\label{fig3}
\end{figure}

\subsection{Comparisons with RK}

In this subsection, we consider doubly-noisy systems constructed from two types of matrices to compare the convergence behavior of RK, RGRK, and RGRK-SA.
We set $\theta\in\{0.5,1\}$ and $N=20$. Note that RGRK with $\theta=1$ reduces to RK equipped with the maximal correction sampling method.

\subsubsection{The simulated data}

We first construct doubly-noisy systems based on overdetermined linear systems with Gaussian matrices. As illustrated in Figure \ref{fig2}, the convergence rates of RGRK and RGRK-SA are comparable, and both of them are faster than RK. While RK and RGRK achieve similar accuracy, RGRK-SA attains the highest accuracy for solving doubly-noisy systems. These results further demonstrate that the signal averaging technique can effectively enhance the accuracy of RGRK.


\begin{figure}[H]
	\centering
	\subfigure{
		\includegraphics[width=0.33\linewidth,height=0.27\textwidth]{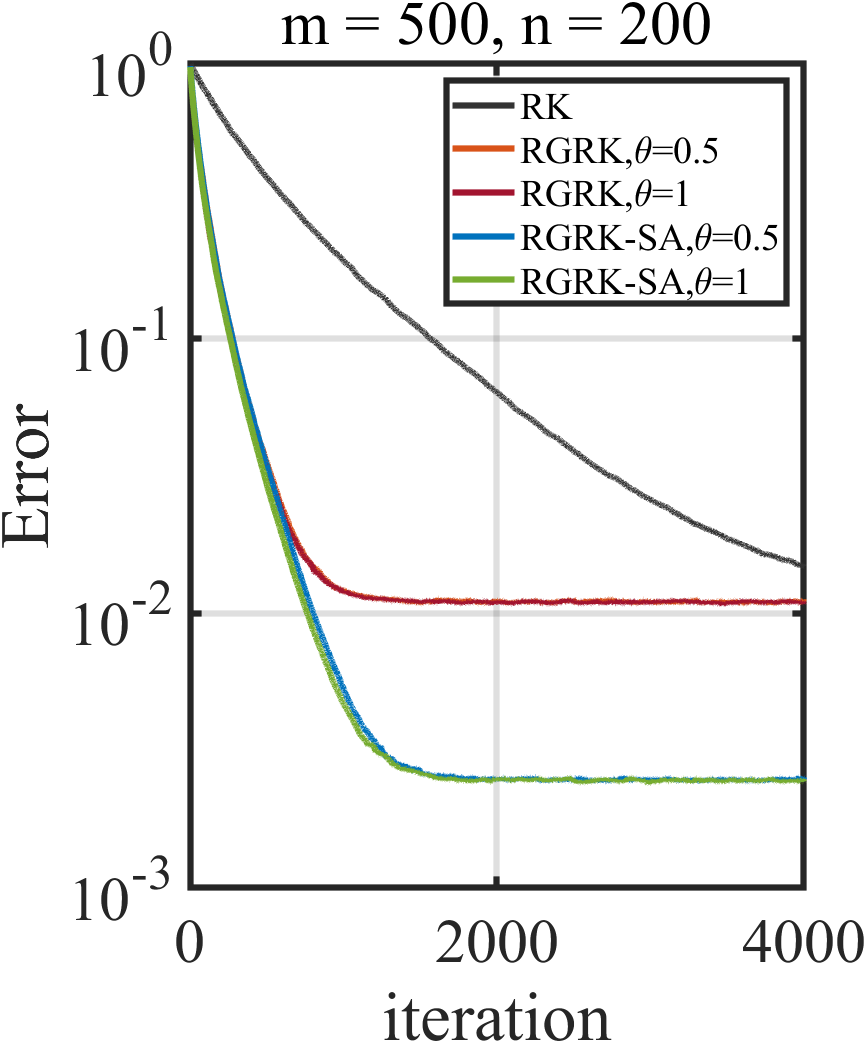}
		\includegraphics[width=0.33\linewidth,height=0.27\textwidth]{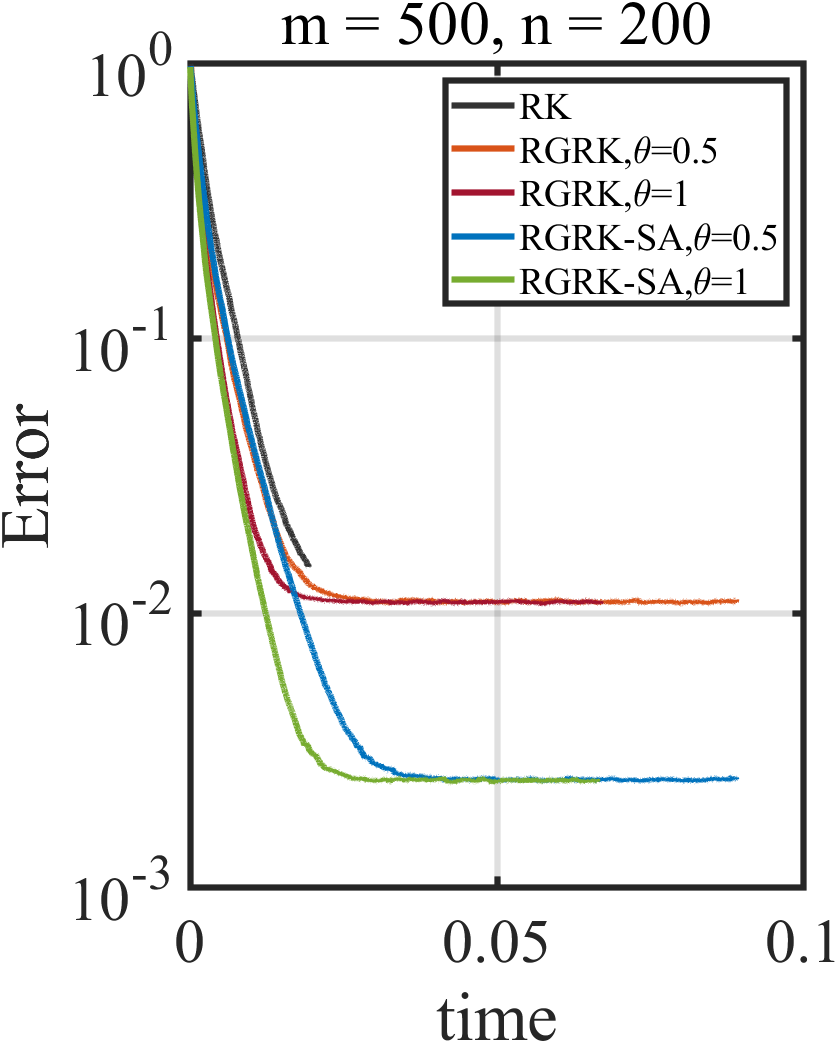}}
	\subfigure{
		\includegraphics[width=0.33\linewidth,height=0.27\textwidth]{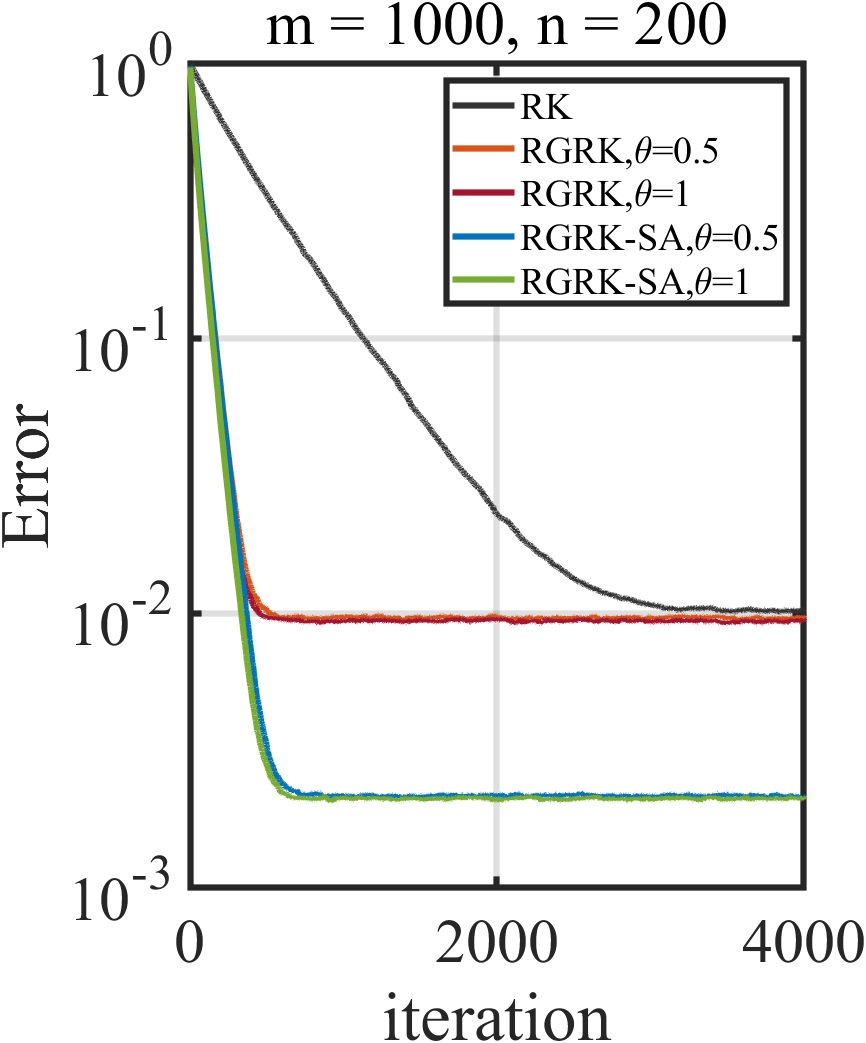}
		\includegraphics[width=0.33\linewidth,height=0.27\textwidth]{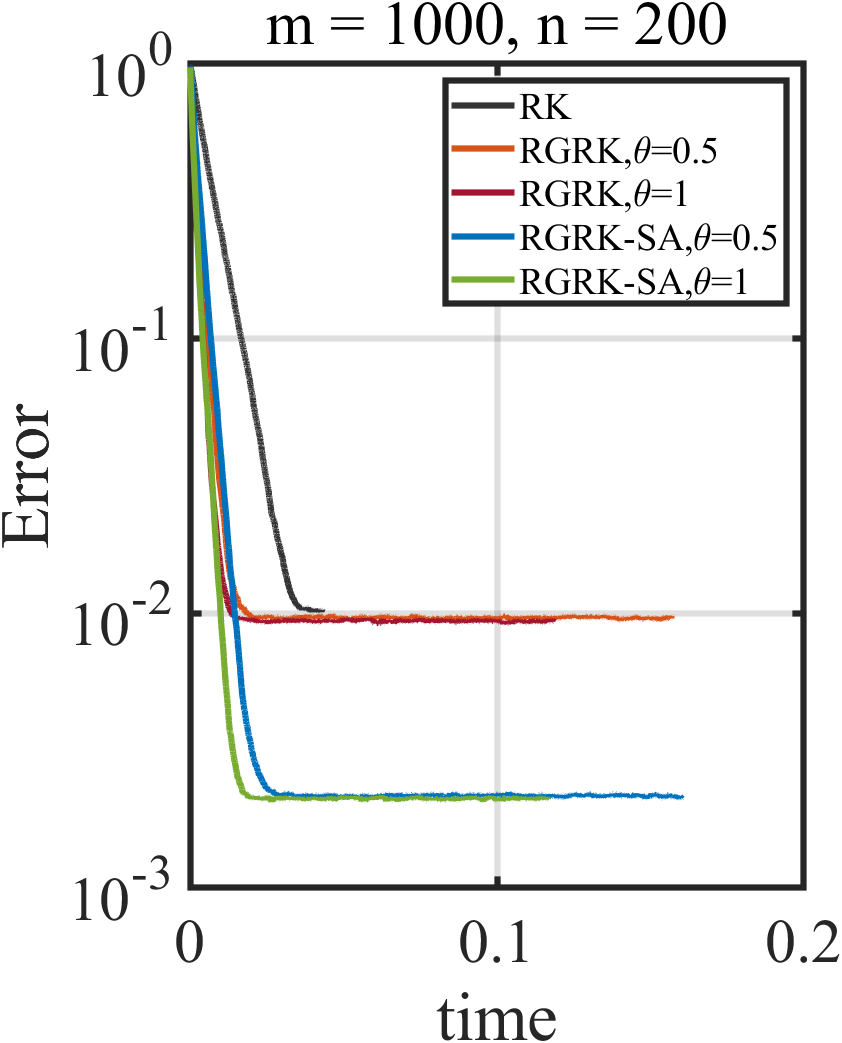}}
	\subfigure{
		\includegraphics[width=0.33\linewidth,height=0.27\textwidth]{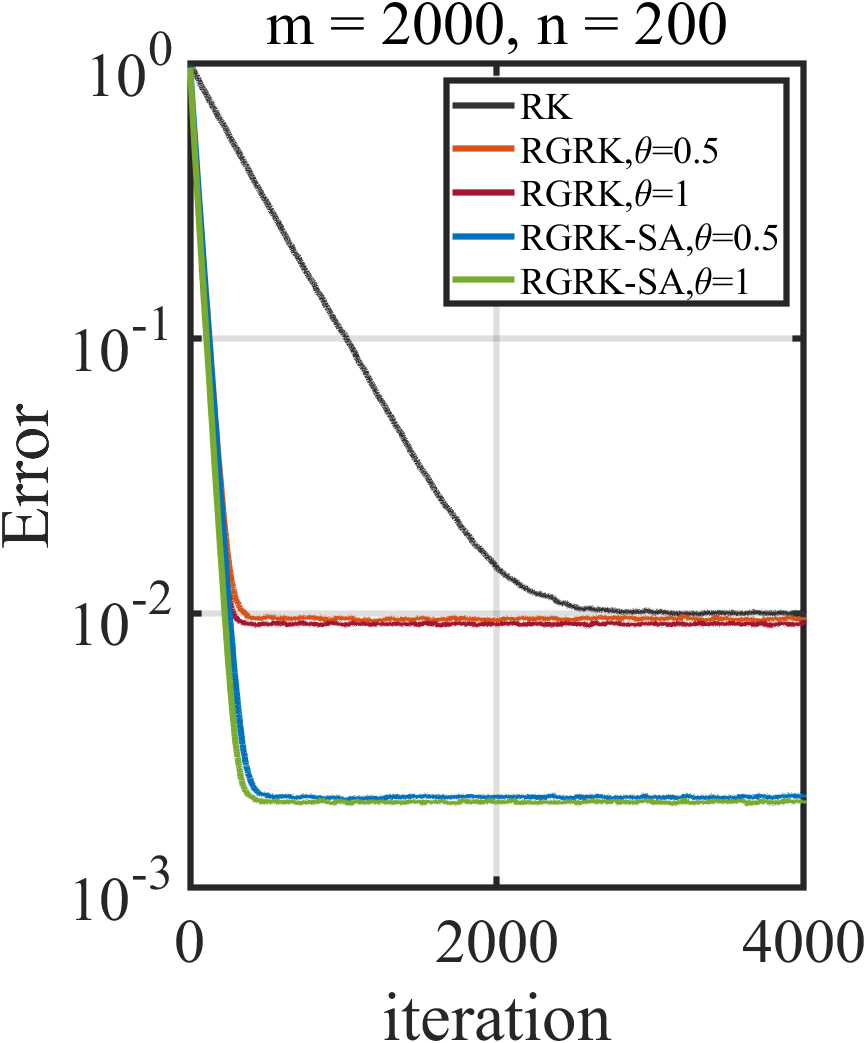}
		\includegraphics[width=0.33\linewidth,height=0.27\textwidth]{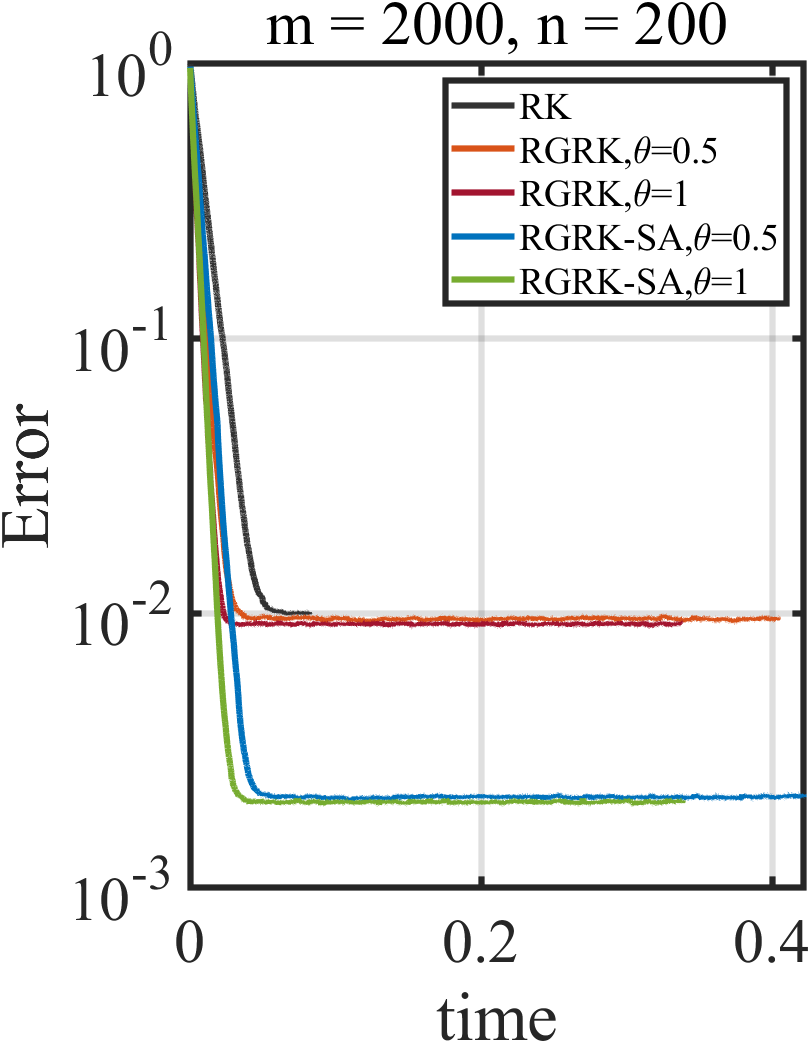}}
	\caption{Relative error and computing time after 4000 iterations of RK, RGRK, and RGRK-SA for simulated linear systems}
	\label{fig2}
\end{figure}

\subsubsection{The real-world data}

Next, we construct doubly-noisy systems based on real-world matrices from the SuiteSparse Matrix Collection in \cite{davis2011university}, summarized in Table \ref{table2}.
The algorithms are terminated when the maxiter iteration reaches $M=2000$. Figure \ref{fig1} plots the median of iterations and computing time (CPU) over 50 trials. It is observed that RGRK-SA outperforms RGRK in terms of both iteration counts and computing times, and both of them are faster than RK in \cite{bergou2024note}.

\begin{figure}[H]
	\centering
	\subfigure[ash958]{
	\includegraphics[width=0.33\linewidth,height=0.27\textwidth]{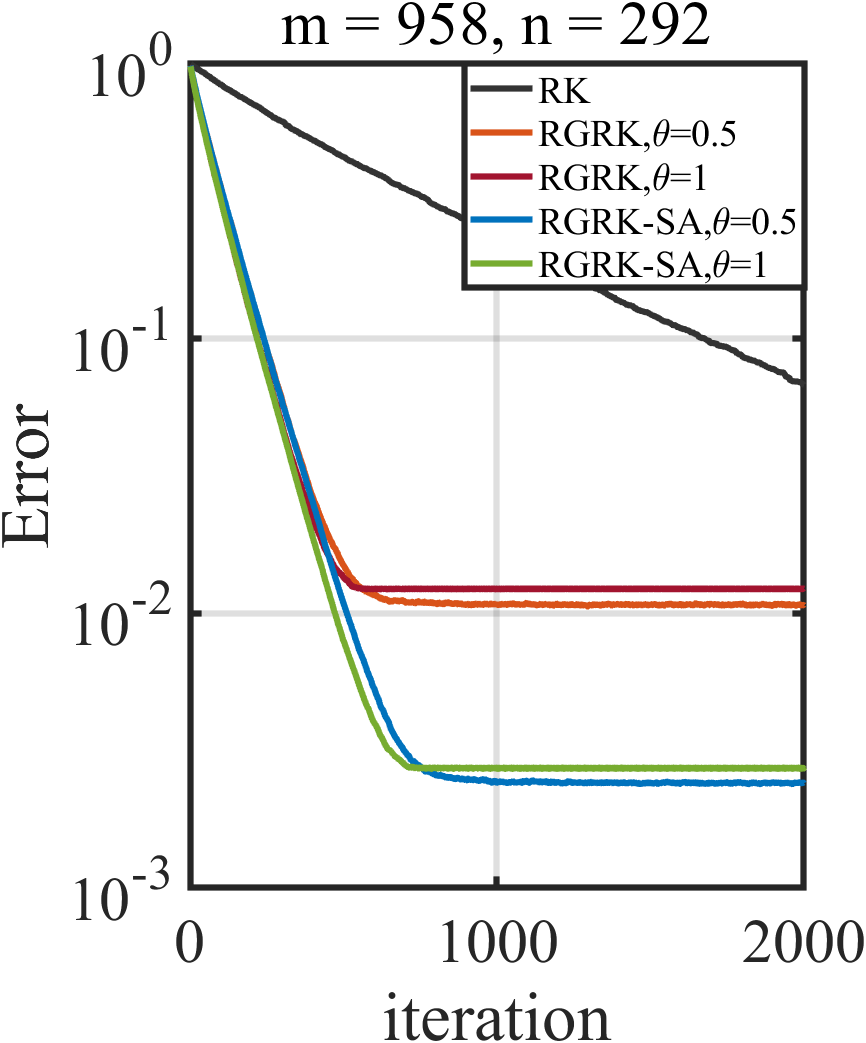}
    \includegraphics[width=0.33\linewidth,height=0.27\textwidth]{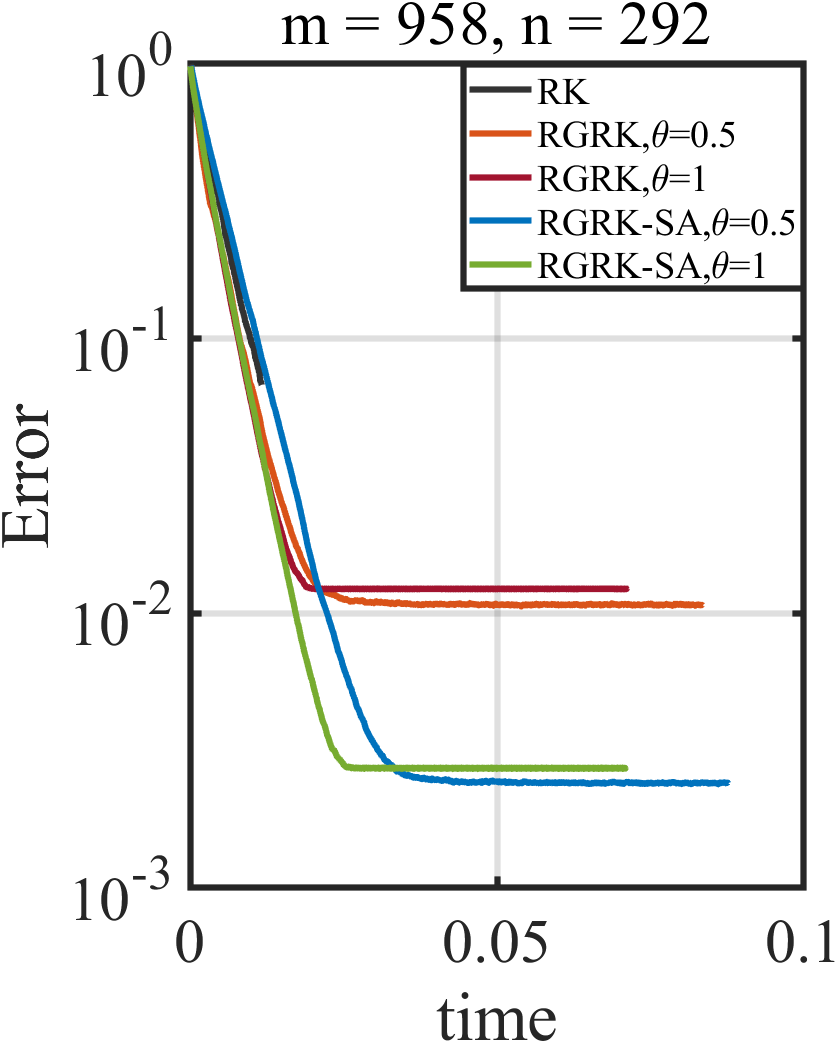}}
	\subfigure[ash219]{
		\includegraphics[width=0.33\linewidth,height=0.27\textwidth]{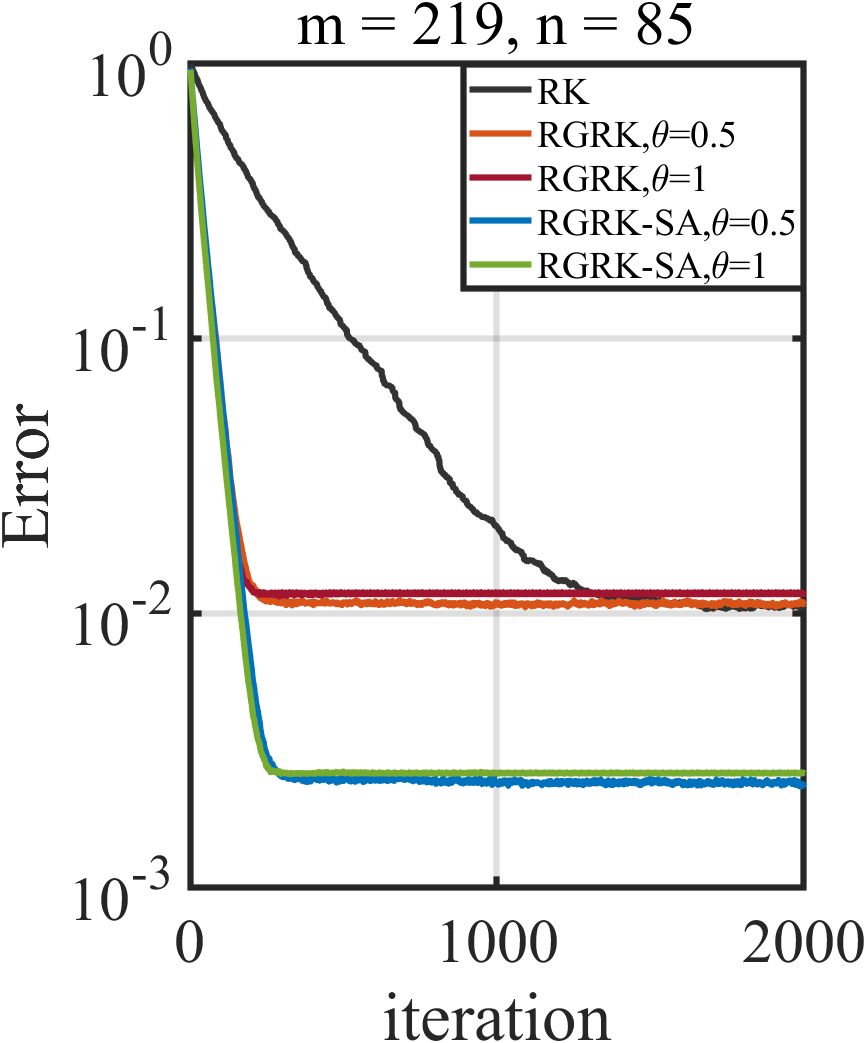}
	    \includegraphics[width=0.33\linewidth,height=0.27\textwidth]{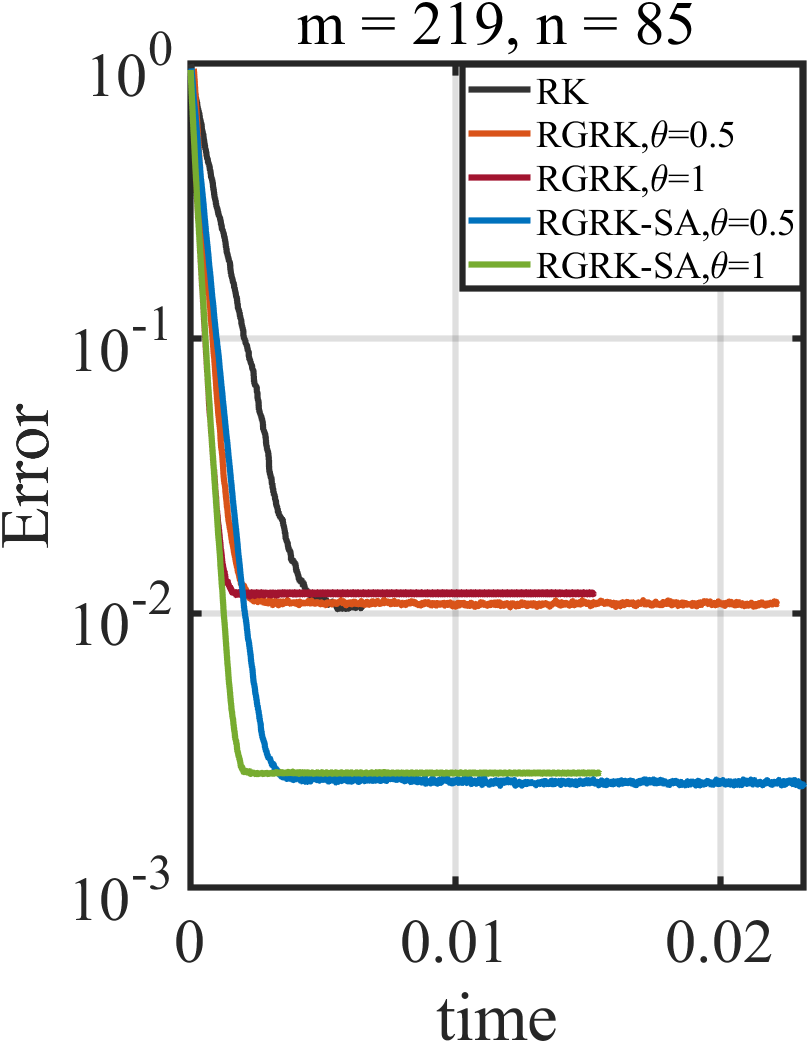}}
	\subfigure[abtaha1]{
		\includegraphics[width=0.33\linewidth,height=0.27\textwidth]{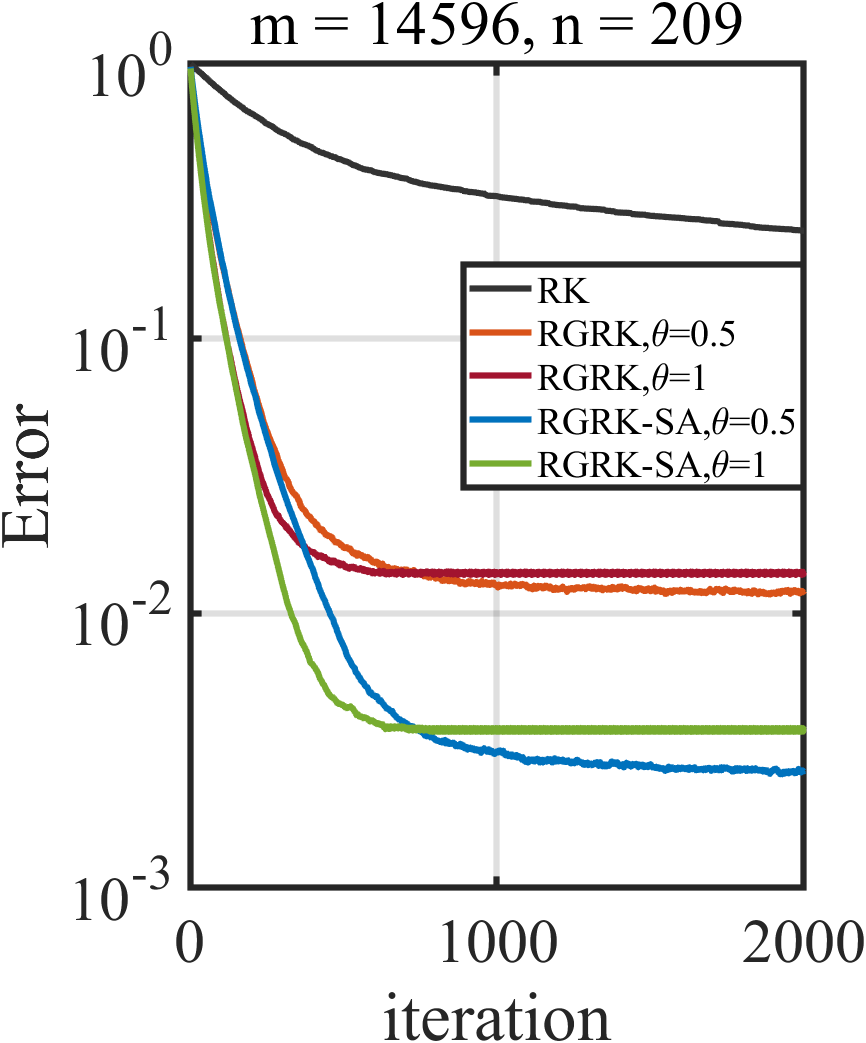}
		\includegraphics[width=0.33\linewidth,height=0.27\textwidth]{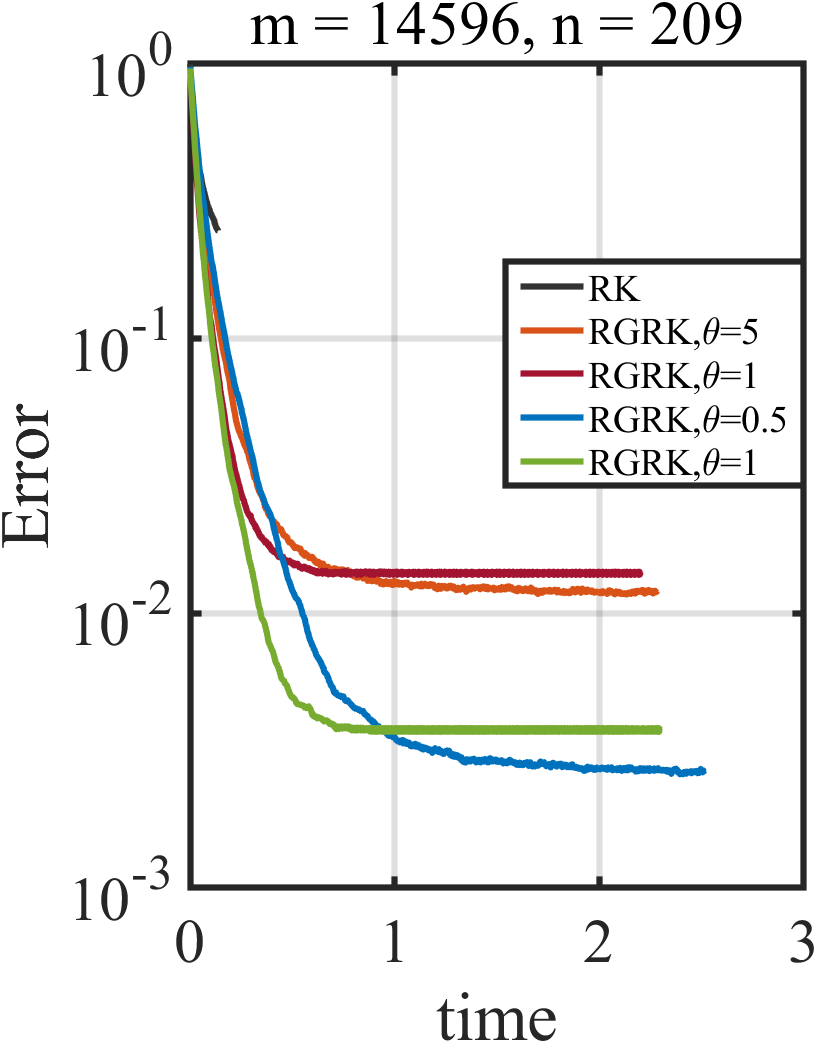}}
	\caption{Relative error and computing time after 2000 iterations of RK, RGRK, and RGRK-SA for real-world linear systems}
	\label{fig1}
\end{figure}

\begin{table}[H]
	\begin{center}
		\begin{minipage}{\textwidth}
			\caption{The information of real-world matrices from SuiteSparse Matrix Collection}
			\label{table2}
			\begin{tabular*}{\textwidth}{@{\extracolsep{\fill}}lccccc@{\extracolsep{\fill}}}
				\toprule
				\multicolumn{2}{c}{ name } &  ash958  & ash219 & abtaha1 &  abtaha2\\
				\midrule
				\multicolumn{2}{c}{ $m\times n$ } & $958\times 292$  & $219\times 85$ & $14596\times 209$ & $37932\times 331$\\	
				\multicolumn{2}{c}{ density } &  0.68\% &  2.35\%  &  1.68\% & 1.09\% \\
				\multicolumn{2}{c}{ cond($A$) } & 3.20  &  3.02 &  12.23 & 1.22 \\
				\bottomrule
			\end{tabular*}
		\end{minipage}
	\end{center}
\end{table}

\section{Conclusions}
\label{sec5}

In this work, we have proposed RGRK and RGRK-SA for solving doubly-noisy linear systems in the absence of exact measurement data, and provide convergence analysis for them. 
In particular, we have introduced a practical and effective scheme for denoising both doubly-noisy and standard noisy linear systems by combining with the signal averaging technique. We have shown that the RGRK-SA method converges to a more exact solution, effectively eliminating the influence of convergence horizon on solution accuracy in both RK and RGRK. Numerical experiments further confirm that the signal averaging technique reduces the effect of noise, demonstrating that our proposed methods achieve higher efficiency and accuracy compared with the standard RK method.

\section*{Acknowledgements}
All authors contributed to the study’s conception and design. The first draft
of the manuscript was written by Lu Zhang and all authors commented on previous versions of the manuscript. All authors read and approve the final manuscript and are all aware of the current submission.

\section*{Funding sources}
This work was funded by the National Natural Science Foundation of China (No.12471300), Independent Innovation Science Fundation of National University of Defense Technology (25-ZZCX-JDZ-14), and Program of China Scholarship Council (Project ID: 202406110004).

\section*{Data statement}
All data that support the findings of this study are included within the article

\footnotesize

\bibliographystyle{plain}
\bibliography{ref}
\end{document}